\documentclass{amsart}
\usepackage{amssymb}
\usepackage{amsmath}
\usepackage{amsfonts}
\usepackage{geometry}

\newtheorem{theorem}{Theorem}
\theoremstyle{plain}
\newtheorem{acknowledgement}{Acknowledgement}

\newtheorem{corollary}{Corollary}

\newtheorem{definition}{Definition}

\newtheorem{lemma}{Lemma}

\newtheorem{proposition}{Proposition}
\newtheorem{remark}{Remark}

\numberwithin{equation}{section}
\input{tcilatex}

\begin{document}
\title[Integrability of linear RDEs and related topics]{Integrability of
(non-)linear rough differential equations and integrals.}
\author{Peter Friz, Sebastian Riedel}
\address{TU\ and WIAS\ Berlin (first and corresponding author,
friz@math.tu-berlin.de), TU\ Berlin (second author)}

\begin{abstract}
Integrability properties of (classical, linear, linear growth) rough
differential equations (RDEs) are considered, the Jacobian of the RDE flow
driven by Gaussian signals being a motivating example. We revisit and extend
some recent ground-breaking work of Cass--Litterer--Lyons in this regard; as
by-product, we obtain a user-friendly "transitivity property" of such
integrability estimates. We also consider rough integrals; as a novel
application, uniform Weibull tail estimates for a class of (random) rough
integrals are obtained. A concrete example arises from the stochastic
heat-equation, spatially mollified by hyper-viscosity, and we can recover
(in fact: sharpen)\ a technical key result of [Hairer,
Comm.PureAppl.Math.64,no.11,(2011),1547--1585].
\end{abstract}

\maketitle
\tableofcontents

\section{Introduction}

Integrability properties of linear rough differential equations (RDEs), and
related topics, driven by Brownian and then a Gaussian rough path (GRP), a
random rough path $\mathbf{X=X}\left( \omega \right) $, have been a serious
difficulty in a variety of recent applications of rough path theory. To wit,
for solutions of linear RDEs one has the typical - and as such sharp -
estimate $O\left( \exp \left( \left( const\right) \times \omega _{\mathbf{x}%
}\left( 0,T\right) \right) \right) $ where $\omega _{\mathbf{x}}\left(
0,T\right) =\left\Vert \mathbf{x}\right\Vert _{p\text{-var;}\left[ 0,T\right]
}^{p}$ denotes some (homogenous) $p$-variation norm (raised to power $p$).
In a Gaussian rough path setting, $\left\Vert \mathbf{X}\right\Vert _{p\text{%
-var;}\left[ 0,T\right] }$ enjoyes Gaussian integrability but as soon as $p>2
$ (the "interesting" case, which covers Brownian and rougher situations) one
has lost all control over moments of such (random) linear RDE solutions. In
a recent work, Cass, Litterer, Lyons \cite{CLL} have overcome a similar
problem, the integrability of the Jacobian of Gaussian RDE flow, as needed
in non-Markovian H\"{o}rmander theory \cite{HP11}.

With these (and some other, cf. below) problems in mind, we revisit the work
of Cass, Litterer, Lyons and propose (what we believe to be) a particularly
user-friendly formulation. We avoid the concept of "\textit{localized }$p$%
\textit{-variation}", as introduced in \cite{CLL}, and work throughout with
a quantity called $N_{\left[ 0,T\right] }\left( \mathbf{x}\right) $. As it
turns out, in many (deterministic) rough path estimates, as obtained in \cite%
{FV10} for instance, one may replace $\omega _{\mathbf{x}}\left( 0,T\right) $
by $N_{\left[ 0,T\right] }\left( \mathbf{x}\right) $. Doing so does not
require to revisit the (technical) proofs of these rough path estimates, but
rather to apply the existing estimates repeatedly on the intervals of a
carefully chosen partition of $\left[ 0,T\right] $. The point is that $N_{%
\left[ 0,T\right] }\left( \mathbf{X}\right) $ enjoyes much better
integrability than $\omega _{\mathbf{X}}\left( 0,T\right) $. Of course, this
does not rule out that for \textit{some} rough paths $\mathbf{x}$, $N_{\left[
0,T\right] }\left( \mathbf{x}\right) \approx \omega _{\mathbf{x}}\left(
0,T\right) $, in agreement with the essentially optimal nature of exisiting
rough path estimates in terms of $\omega _{\mathbf{x}}\left( 0,T\right) $.
For instance, both quantities will scale like $\lambda $ when $p=2$ and $%
\mathbf{x}$ is the \textit{pure-area rough path}, dilated by $\lambda >>1$.
Differently put, the point is that $N_{\left[ 0,T\right] }\left( \mathbf{X}%
\left( \omega \right) \right) $ will be smaller than $\omega _{\mathbf{X}%
\left( \omega \right) }\left( 0,T\right) $ for \textit{most} realizations of 
$\mathbf{X}\left( \omega \right) $.

The consequent focus on $N_{\left[ 0,T\right] }$ rather than "localized $p$%
-variation" aside, let us briefly enlist our contributions relative to \cite%
{CLL}. 

\begin{itemize}
\item[(i)] A technical condition "$p>q\left[ p\right] $" is removed; this
shows that the Cass, Litterer, Lyons results are valid assuming only
"complementary Young regularity of the Cameron-Martin space", i.e. $%
H\hookrightarrow C^{q-var}$ where $1/p+1/q>1$ and sample paths have finite $%
p $-variation, a natural condition, in particular in the context of
Malliavin calculus, whose importance was confirmed in a number of papers, 
\cite{friz-victoir-2007-gauss}, \cite{FO10}, \cite{CFV}, \cite{CF}, see also 
\cite{FV10}.

\item[(ii)] Their technical main result, Weibull tails of $N_{\left[ 0,T%
\right] }\left( \mathbf{X}\right) $ with shape parameter $2/q$, here $%
\mathbf{X}$ is a Gaussian rough path, remains valid for general rough paths
obtained as image of $\mathbf{X}$ under \textit{locally linear maps on
(rough) path space}. (This random rough paths may be far from Gaussian:
examples of locally linear maps are given by rough integration and (solving)
rough differential equations.)

\item[(iii)] The arguments are adapted to deal with (random) linear (and
also linear growth) rough differential equations (the solution maps here are
not locally linear!) driven by $\mathbf{X}\left( \omega \right) $. As above,
it suffices that $\mathbf{X}$ is the locally linear image of a Gaussian
rough path.
\end{itemize}

We conclude with two applications. First, we show how to recover log-Weibull
tails for $\left\vert J\right\vert $, the Jacobian of a Gaussian RDE flow.
(Afore-mentioned extended validity and some minor sharpening of the norm of $%
\left\vert J\right\vert $ aside, this was the main result of \cite{CLL}.)
Our point here is that \cite{CLL} use somewhat involved (known) explicit
estimates for the $J$ in terms of the Gaussian driving signal. In contrast,
our "user-friendly"\ formulation allows for a simple step-by-step approach:
recall that $J$ solves $dJ=J\,dM$ where $M$ is a non-Gaussian driving rough
path, obtained by solving an RDE / performing a rough integration. Since $M$
is the locally linear image of a Gaussian rough path, we can immediately
appeal to (iii). As was pointed out recently by \cite{HP11}, such estimates
are - in combination with a Norris lemma for rough paths - the key to a
non-Markovian Hoermander and then ergodic theory.\newline
Secondly, as a novel application, we consider (random) rough integrals of
the form $\int G\left( X\,\right) \,d\mathbf{X}$, with $G\in Lip^{\gamma -1}$%
,$\gamma >p$ and establish Weibull tails with shape parameter $2/q$, uniform
over classes of Gaussian process whose covariance satisfies a uniform
variational estimate. A special case arises when $\mathbf{X}=\mathbf{X}%
^{\epsilon }$ is taken, independently in each component, as solution to the
stochastic heat equation on the 1D torus, $\dot{u}=u_{xx}+\dot{W}$ with
hyper-viscosity term $\epsilon u_{xxxx}$ - as function of the space
variable, for fixed time. Complementary Young regularity is seen to hold
with $p>2$ and and $q=1$ and we so obtain (and in fact, improve from
exponential to Gaussian integrability) the uniform in $\epsilon $
integrability estimate \cite{H11}, Theorem 5.1, a somewhat central technical
result whose proof encompasses almost a third of that paper.

\section{Basis definitions}

\begin{definition}
Let $\omega $ be a control. For $\alpha >0$ and $\left[ s,t\right] \subset %
\left[ 0,1\right] $ we set%
\begin{eqnarray*}
\tau _{0}\left( \alpha \right) &=&s \\
\tau _{i+1}\left( \alpha \right) &=&\inf \left\{ u:\omega \left( \tau
_{i},u\right) \geq \alpha ,\tau _{i}\left( \alpha \right) <u\leq t\right\}
\wedge t
\end{eqnarray*}%
and define%
\begin{equation*}
N_{\alpha ,\left[ s,t\right] }\left( \omega \right) =\sup \left\{ n\in 
\mathbb{N\cup }\left\{ 0\right\} :\tau _{n}\left( \alpha \right) <t\right\} .
\end{equation*}%
When $\omega $ arises from a (homogenous) $p$-variation norm of a ($p$%
-rough) path, such as $\omega _{\mathbf{x}}=\left\Vert \mathbf{x}\right\Vert
_{p\text{-var;}\left[ \cdot ,\cdot \right] }^{p}$ or $\bar{\omega}_{\mathbf{x%
}}:=\left\vert \left\vert \left\vert \mathbf{x}\right\vert \right\vert
\right\vert _{p\text{-var;}\left[ \cdot ,\cdot \right] }^{p}$, detailed
definitions are give later in the text, we shall also write%
\begin{equation*}
N_{\alpha ,\left[ s,t\right] }\left( \mathbf{x}\right) :=N_{\alpha ,\left[
s,t\right] }\left( \omega _{\mathbf{x}}\right) \text{ and }\bar{N}_{\alpha ,%
\left[ s,t\right] }\left( \mathbf{x}\right) :=N_{\alpha ,\left[ s,t\right]
}\left( \bar{\omega}_{\mathbf{x}}\right) \text{.}
\end{equation*}
\end{definition}

In fact, we will be in a situation where $C^{-1}\bar{\omega}_{\mathbf{x}%
}\leq \omega _{\mathbf{x}}\leq C\bar{\omega}_{\mathbf{x}}$ for some constant 
$C$ which entails (cf. Lemma \ref{lemma_transitivity_of_N} below)%
\begin{equation*}
\bar{N}_{\alpha C,\left[ \cdot ,\cdot \right] }\left( \mathbf{x}\right) \leq
N_{\alpha ,\left[ \cdot ,\cdot \right] }\left( \mathbf{x}\right) \leq \bar{N}%
_{\alpha /C,\left[ \cdot ,\cdot \right] }\left( \mathbf{x}\right) .
\end{equation*}%
Furthermore, the precise value of $\alpha >0$ will not matter (c.f. Lemma %
\ref{lemma_Nalpha_Nbeta_estimate} below) so that a factor $C$ or $1/C$ is
indeed inconsequential; effectively, this means that one can switch between $%
N$ and $\bar{N}$ as one pleases.

We now study the scaling of $N_{\alpha }$. Note that $N_{\alpha ,\left[ s,t%
\right] }\left( \omega \right) \searrow 0$ for $\alpha \nearrow \infty $.

\begin{lemma}
\label{lemma_scaling_N_omega}Let $\omega $ be a control and $\lambda >0$.
Then $\left( s,t\right) \mapsto \lambda \omega \left( s,t\right) $ is again
a control and for all \thinspace $s<t$,%
\begin{equation*}
N_{\alpha ,\left[ s,t\right] }\left( \lambda \omega \right) =N_{\alpha
/\lambda ,\left[ s,t\right] }\left( \omega \right) .
\end{equation*}
\end{lemma}

\begin{proof}
Follows directly from the definition.
\end{proof}

\begin{lemma}
\label{lemma_transitivity_of_N}Let $\omega _{1},\omega _{2}$ be two
controls, $s<t$ and $\alpha >0$. Assume that $\omega _{1}\left( u,v\right)
\leq C\omega _{2}\left( u,v\right) $ holds whenever $\omega _{2}\left(
u,v\right) \leq \alpha $ for a constant $C$. Then $N_{C\alpha ,\left[ s,t%
\right] }\left( \omega _{1}\right) \leq N_{\alpha ,\left[ s,t\right] }\left(
\omega _{2}\right) $.
\end{lemma}

\begin{proof}
It suffices to consider the case $C=1$, the general case follows by the
scaling of $N$. Set%
\begin{eqnarray*}
\tau _{0}^{j}\left( \alpha \right) &=&s \\
\tau _{i+1}^{j}\left( \alpha \right) &=&\inf \left\{ u:\omega _{j}\left(
\tau _{i}^{j},u\right) \geq \alpha ,\tau _{i}^{j}\left( \alpha \right)
<u\leq t\right\} \wedge t
\end{eqnarray*}%
for $j=1,2$. It suffices to show that $\tau _{i}^{2}\leq \tau _{i}^{1}$
holds for every $i\in \mathbb{N}$. By induction over $i$: For $i=0$ this is
clear. If $\tau _{i}^{2}\leq \tau _{i}^{1}$ for some fixed $i$,%
\begin{equation*}
\omega _{1}\left( \tau _{i}^{1},u\right) \leq \omega _{2}\left( \tau
_{i}^{1},u\right) \leq \omega _{2}\left( \tau _{i}^{2},u\right)
\end{equation*}%
whenever $\omega _{2}\left( \tau _{i}^{2},u\right) \leq \alpha $. Hence%
\begin{equation*}
\inf_{u}\left\{ \omega _{2}\left( \tau _{i}^{2},u\right) \geq \alpha
\right\} \leq \inf_{u}\left\{ \omega _{1}\left( \tau _{i}^{2},u\right) \geq
\alpha \right\}
\end{equation*}%
and therefore $\tau _{i+1}^{2}\leq \tau _{i+1}^{1}$.
\end{proof}

\begin{lemma}
\label{lemma_Nalpha_Nbeta_estimate} Let $\omega $ be a control and $0<\alpha
\leq \beta $. Then%
\begin{equation*}
N_{\alpha ,\left[ s,t\right] }\left( \omega \right) \leq \frac{\beta }{%
\alpha }\left( 2N_{\beta ,\left[ s,t\right] }\left( \omega \right) +1\right)
.
\end{equation*}
\end{lemma}

\begin{proof}
Set%
\begin{equation*}
\omega _{\alpha }\left( s,t\right) :=\sup_{\substack{ \left( t_{i}\right)
=D\subset \left[ s,t\right]  \\ \omega (t_{i},t_{i+1})\leq \alpha }}%
\sum_{t_{i}}\omega (t_{i},t_{i+1}).
\end{equation*}%
We clearly have $\omega _{\alpha }\left( s,t\right) \leq \omega _{\beta
}\left( s,t\right) $ and%
\begin{equation*}
\alpha N_{\alpha ,\left[ s,t\right] }\left( \omega \right)
=\sum_{i=0}^{N_{\alpha ,\left[ s,t\right] }\left( \omega \right) -1}\omega
(\tau _{i}\left( \alpha \right) ,\tau _{i+1}\left( \alpha \right) )\leq
\omega _{\alpha }(s,t).
\end{equation*}%
Finally, Proposition 4.6 in \cite{CLL} shows that $\omega _{\beta }\left(
s,t\right) \leq \left( 2N_{\beta ,\left[ s,t\right] }\left( \omega \right)
+1\right) \beta $. (Strictly speaking, Proposition 4.6 is formulated for a
particular control $\omega $, namely the control induced by the $p$%
-variation of a rough path. However, the proof only uses general properties
of control functions and the conclusion remains valid.)
\end{proof}

Let $\mathbf{x\colon }[0,T]\rightarrow G^{N}\left( \mathbb{R}^{d}\right) $
be a path. In the whole section, $\left\Vert \cdot \right\Vert _{p-var}$
denotes the $p$-variation norm for such paths induced by the
Carnot-Caratheodory metric; \cite{FV10}. Set $\omega _{\mathbf{x}%
}(s,t)=\left\Vert \mathbf{x}\right\Vert _{p-var;[s,t]}^{p}$ and $N_{\alpha ,%
\left[ s,t\right] }\left( \mathbf{x}\right) =N_{\alpha ,\left[ s,t\right]
}\left( \omega _{\mathbf{x}}\right) $ (the fact that $\omega _{\mathbf{x}}$
is indeed a control is well-known; c.f. \cite{FV10}).

\begin{lemma}
\label{lemma_pvar_dom_by_N}For any $\alpha >0$,%
\begin{equation*}
\left\Vert \mathbf{x}\right\Vert _{p-var;[s,t]}\leq \alpha ^{1/p}\left(
N_{\alpha ,\left[ s,t\right] }\left( \mathbf{x}\right) +1\right)
\end{equation*}
\end{lemma}

\begin{proof}
Let $u=u_{0}<u_{1}<\ldots <u_{m}=v$. Note that%
\begin{equation*}
\left\Vert \mathbf{x}_{u,v}\right\Vert ^{p}=\left\Vert \mathbf{x}%
_{u,u_{1}}\otimes \mathbf{x}_{u_{1},u_{2}}\otimes \ldots \otimes \mathbf{x}%
_{u_{m-1},v}\right\Vert ^{p}\leq m^{p-1}\sum_{i=0}^{m-1}\left\Vert \mathbf{x}%
_{u_{i},u_{i+1}}\right\Vert ^{p}.
\end{equation*}%
Let $D$ be a dissection of $\left[ s,t\right] $ and $\left( \tau _{j}\right)
_{j=0}^{N_{\alpha ,\left[ s,t\right] }\left( \mathbf{x}\right) }=\left( \tau
_{j}\left( \alpha \right) \right) _{j=0}^{N_{\alpha ,\left[ s,t\right]
}\left( \mathbf{x}\right) }$. Set $\bar{D}=D\cup \left( \tau _{j}\right)
_{j=0}^{N_{\alpha ,\left[ s,t\right] }\left( \mathbf{x}\right) }$. Then,%
\begin{eqnarray*}
\sum_{t_{i}\in D}\left\Vert \mathbf{x}_{t_{i},t_{i+1}}\right\Vert ^{p} &\leq
&\left( N_{\alpha ,\left[ s,t\right] }\left( \mathbf{x}\right) +1\right)
^{p-1}\sum_{\bar{t}_{i}\in \bar{D}}\left\Vert \mathbf{x}_{\bar{t}_{i},\bar{t}%
_{i+1}}\right\Vert ^{p} \\
&\leq &\left( N_{\alpha ,\left[ s,t\right] }\left( \mathbf{x}\right)
+1\right) ^{p-1}\sum_{j=0}^{N_{\alpha ,\left[ s,t\right] }\left( \mathbf{x}%
\right) }\left\Vert \mathbf{x}\right\Vert _{p-var;\left[ \tau _{j},\tau
_{j+1}\right] }^{p} \\
&\leq &\left( N_{\alpha ,\left[ s,t\right] }\left( \mathbf{x}\right)
+1\right) ^{p}\alpha .
\end{eqnarray*}%
Taking the supremum over all partitions shows the claim.
\end{proof}

\section{Cass, Litterer and Lyons revisited}

The basic object is a continuous $d$-dimensional Gaussian process, say $X$,
realized as coordinate process on the (not-too abstract) Wiener space $%
\left( E,\mathcal{H},\mu \right) $ where $E=C\left( \left[ 0,T\right] ,%
\mathbb{R}^{d}\right) $ equipped with $\mu $ is a Gaussian measure s.t. $X$
has zero-mean, independent components and that $V_{\rho \text{-var}}\left( R,%
\left[ 0,T\right] ^{2}\right) $, the $\rho $-variation in 2D sense of the
covariance $R$ of $X$, is finite for some $\rho \in \lbrack 1,2)$. From \cite%
[Theorem 15.33]{FV10} it follows that we can lift the sample paths of $X$ to 
$p$-rough paths for any $p>2\rho $ and we denote this process by $\mathbf{X}$%
, called the \emph{enhanced Gaussian process}. We also assume that the
Cameron-Martin space $\mathcal{H}$ has \emph{complementary Young regularity}
in the sense that $\mathcal{H}$ embeds continuously in $C^{q\text{-var}%
}\left( \left[ 0,T\right] ,\mathbb{R}^{d}\right) $ with $\frac{1}{p}+\frac{1%
}{q}>1$. Note $q\leq p$ for $\mu $ is supported on the paths of finite $p$%
-variation. There are many examples of such a situation \cite{FV10}, let us
just note that fractional Brownian motion (fBM) with Hurst parameter $H>1/4$
falls in this class of Gaussian rough paths.

In this section, we present, in a self-contained fashion, the results \cite%
{CLL}. In fact, we present a slightly modified argument which avoids the
technical condition "$p>q\left[ p\right] $" made in \cite[Theorem 6.2,
condition (3)]{CLL} (this still applies to fBM with~$H>1/4$ but causes some
discontinuities in the resulting estimates when $H$ crosses the barrier $1/3$%
). Our argument also gives a unified treatment for all $p$ thereby
clarifying the structure of the proof\ (in \cite[Theorem 6.2]{CLL} the cases 
$\left[ p\right] =2,3$ are treated separately "by hand"). That said, we
clearly follow \cite{CLL} in their ingenious use of Borell's inequality.

In the whole section, if not stated otherwise, for a $p$-rough path $\mathbf{%
x}$, set%
\begin{equation*}
\left\vert \left\vert \left\vert \mathbf{x}\right\vert \right\vert
\right\vert _{p-var;[s,t]}:=\left( \sum_{k=1}^{\left[ p\right] }\left\Vert 
\mathbf{x}^{\left( k\right) }\right\Vert _{p/k-var;\left[ s,t\right]
}^{p/k}\right) ^{1/p}.
\end{equation*}%
Then $\left\vert \left\vert \left\vert \mathbf{\cdot }\right\vert
\right\vert \right\vert _{p-var}$ is a homogeneous rough path norm. Recall
that, as a consequence of Theorem 7.44 in \cite{FV10}, the norms $\left\vert
\left\vert \left\vert \mathbf{\cdot }\right\vert \right\vert \right\vert
_{p-var}$ and $\left\vert \left\vert \mathbf{\cdot }\right\vert \right\vert
_{p-var}$ are equivalent, hence there is a constant $C$ such that%
\begin{equation}
\frac{1}{C}\left\vert \left\vert \left\vert \mathbf{\cdot }\right\vert
\right\vert \right\vert _{p-var}\leq \left\vert \left\vert \mathbf{\cdot }%
\right\vert \right\vert _{p-var}\leq C\left\vert \left\vert \left\vert 
\mathbf{\cdot }\right\vert \right\vert \right\vert _{p-var}.
\label{eqn_equiv_hom_rp_norms}
\end{equation}%
The map $\left( s,t\right) \mapsto \bar{\omega}_{\mathbf{x}}\left(
s,t\right) =\left\vert \left\vert \left\vert \mathbf{x}\right\vert
\right\vert \right\vert _{p-var;[s,t]}^{p}$ is a control and we set $\bar{N}%
_{\alpha ,\left[ s,t\right] }\left( \mathbf{x}\right) =N_{\alpha ,\left[ s,t%
\right] }\left( \bar{\omega}_{\mathbf{x}}\right) $.

\begin{lemma}
\label{lemma_Aa_pos_meas}Assume that $\mathcal{H}$ has complementary Young
regularity to $X$. Then for any $a>0$, the set%
\begin{equation*}
A_{a}=\left\{ \left\vert \left\vert \left\vert \mathbf{X}\right\vert
\right\vert \right\vert _{p-var;[0,T]}<a\right\}
\end{equation*}%
has positive $\mu $-measure. Moreover, if $M\geq V_{\rho -\text{var}}\left(
R;\left[ 0,T\right] ^{2}\right) $, we have the lower bound%
\begin{equation*}
\mu \left\{ |||\mathbf{X|||}_{p-var;\left[ 0,T\right] }<a\right\} \geq 1-%
\frac{C}{\exp \left( a\right) }
\end{equation*}%
where $C$ is a constant only depending on $\rho ,p$ and $M$.
\end{lemma}

\begin{proof}
The support theorem for Gaussian rough paths (\cite[Theorem 15.60]{FV10})
shows that%
\begin{equation*}
\text{supp}\left[ \mathbf{X}_{\ast }\mu \right] =\overline{S_{\left[ p\right]
}\left( \mathcal{H}\right) }
\end{equation*}%
holds for $p\in (2\rho ,4)$. Hence every neighbourhood of the zero-path has
positive measure which is the first statement. The general case follows from
the a.s. estimate%
\begin{equation}
|||S_{[p^{\prime }]}\left( \mathbf{X}\right) |||_{p^{\prime }-var}\leq
|||S_{[p^{\prime }]}\left( \mathbf{X}\right) |||_{p-var}\leq C_{p,p^{\prime
}}|||\mathbf{X|||}_{p-var}  \label{eqn_lipschitz_lyons_lift}
\end{equation}%
which holds for every $p\leq p^{\prime }$, c.f. \cite{FV10}, Theorem 9.5.
For the lower bound, recall that from \cite{FV10}, Theorem 15.33 one can
deduce that%
\begin{equation}
E\left( \exp |||\mathbf{X|||}_{p-var;\left[ 0,T\right] }\right) \leq C
\label{eqn_moments_p_var_unif_bdd}
\end{equation}%
for $p\in (2\rho ,4)$ where $C$ only depends on $\rho ,p$ and $M$. Using $%
\left( \ref{eqn_lipschitz_lyons_lift}\right) $ shows that this actually
holds for every $p>2\rho $. Finally, by Chebychev's inequality,%
\begin{equation*}
\mu \left\{ |||\mathbf{X|||}_{p-var;\left[ 0,T\right] }<a\right\} \geq 1-%
\frac{C}{\exp \left( a\right) }.
\end{equation*}
\end{proof}

In the next theorem we cite the famous isoperimetric inequality due to C.
Borell (for a proof c.f. \cite[Theorem 4.3]{Le}).

\begin{theorem}[Borell]
\label{theorem_borell}Let $\left( E,\mathcal{H},\mu \right) $ be an abstract
Wiener space and $\mathcal{K}$ denote the unit ball in $H$. If $A\subset E$
is a Borell set with positive measure, then for every $r\geq 0$%
\begin{equation*}
\mu \left( A+r\mathcal{K}\right) \geq \Phi \left( \Phi ^{-1}\left( \mu
\left( A\right) \right) +r\right)
\end{equation*}%
where $\Phi $ is the cumulative distribution function of a standard normal
random variable, i.e. $\Phi =\left( 2\pi \right) ^{-1/2}\int_{-\infty
}^{\cdot }\exp \left( -x^{2}/2\right) \,dx$.
\end{theorem}

\begin{corollary}
\label{cor_conseq_borell}Let $f,g\colon E\rightarrow \left[ 0,\infty \right] 
$ be measurable maps and $a,\sigma >0$ such that%
\begin{equation*}
A_{a}:=\left\{ x:f\left( x\right) \leq a\right\}
\end{equation*}%
has positive measure and let $\hat{a}\leq \Phi ^{-1}\mu \left( A_{a}\right) $%
. Assume furthermore that there exists a null-set $N$ such that for all $%
x\in N^{c}$ and $h\in \mathcal{H}:$%
\begin{equation*}
f\left( x-h\right) \leq a\Rightarrow \sigma \left\Vert h\right\Vert _{%
\mathcal{H}}\geq g\left( x\right) .
\end{equation*}%
Then $g$ has a Gauss tail; more precisely, for all $r>0$,%
\begin{equation*}
\mu \left( \left\{ x:g\left( x\right) >r\right\} \right) \leq \exp \left( -%
\frac{\left( \hat{a}+\frac{r}{\sigma }\right) ^{2}}{2}\right) .
\end{equation*}
\end{corollary}

\begin{proof}
W.l.o.g. $\sigma =1$. Then%
\begin{eqnarray*}
\left\{ x:g\left( x\right) \leq r\right\} &=&\bigcup_{h\in r\mathcal{K}%
}\left\{ x:\left\Vert h\right\Vert _{\mathcal{H}}\geq g\left( x\right)
\right\} \\
&\supset &\bigcup_{h\in r\mathcal{K}}\left\{ x:f\left( x-h\right) \leq
a\right\} \\
&=&\bigcup_{h\in r\mathcal{K}}\left\{ x+h:f\left( x\right) \leq a\right\} \\
&=&A_{a}+r\mathcal{K}\text{.}
\end{eqnarray*}%
By Theorem \ref{theorem_borell},%
\begin{equation*}
\mu \left( \left\{ x:g\left( x\right) >r\right\} \right) \leq \mu \left(
\left\{ A_{a}+r\mathcal{K}\right\} ^{c}\right) \leq \bar{\Phi}\left( \hat{a}%
+r\right)
\end{equation*}%
where $\bar{\Phi}=1-\Phi $. The claim follows from the standard estimate $%
\bar{\Phi}\left( r\right) \leq \exp \left( -r^{2}/2\right) $.
\end{proof}

\begin{proposition}
\label{prop_cll_theorem}Let $X$ be a continuous $d$-dimensional Gaussian
process, realized as coordinate process on $\left( E,\mathcal{H},\mu \right) 
$ where $E=C\left( \left[ 0,T\right] ,\mathbb{R}^{d}\right) $ equipped with $%
\mu $ is a Gaussian measure s.t. $X$ has zero-mean, independent components
and that the covariance $R$ of $X$ has finite $\rho $-variationen for some $%
\rho \in \lbrack 1,2).$ Let $\mathbf{X}$ be its enhanced Gaussian process
with sample paths in a $p$-rough paths space, $p>2\rho $. Assume $\mathcal{H}
$ has complementary Young regularity, so that Cameron--Martin paths enjoy
finite $q$-variation regularity, $q\leq p$ and $\frac{1}{p}+\frac{1}{q}>1$.
Then there exists a set $\tilde{E}\subset E$ of full measure with the
following property: If%
\begin{equation}
\left\vert \left\vert \left\vert \mathbf{X}\left( \omega -h\right)
\right\vert \right\vert \right\vert _{p-var;\left[ 0,T\right] }\leq \alpha
^{1/p}  \label{eqn_assumption_cll}
\end{equation}%
for all $\omega \in \tilde{E}$, $h\in \mathcal{H}$ and some $\alpha >0$ then%
\begin{equation*}
C\left\vert h\right\vert _{q-var;\left[ 0,T\right] }\geq \alpha ^{1/p}\left( 
\bar{N}_{\beta ,\left[ 0,T\right] }\left( \mathbf{X}\left( \omega \right)
\right) \right) ^{1/q}
\end{equation*}%
where $\beta =2^{p}\left[ p\right] \alpha $ and $C$ depends only on $p$ and $%
q$.
\end{proposition}

\begin{proof}
Set%
\begin{equation*}
\tilde{E}=\left\{ \omega :T_{h}\left( \mathbf{X}\left( \omega \right)
\right) =\mathbf{X}\left( \omega +h\right) \text{ for all }h\in \mathcal{H}%
\right\} .
\end{equation*}%
From \cite[Lemma 15.58]{FV10} we know that $\tilde{E}$ has full measure.
Define the random partition $\left( \tau _{i}\right) _{i=0}^{\infty }=\left(
\tau _{i}\left( \beta \right) \right) _{i=0}^{\infty }$ for the control $%
\bar{\omega}_{\mathbf{X}}$. Let $h\in \mathcal{H}$ and assume that $\left( %
\ref{eqn_assumption_cll}\right) $ holds. We claim that there is a constant $%
C_{p,q}$ such that%
\begin{equation}
C_{p,q}\left\vert h\right\vert _{q-var;\left[ \tau _{i},\tau _{i+1}\right]
}\geq \alpha ^{1/p}\quad \text{for all }i=0,\ldots ,\bar{N}_{\beta ,\left[
0,T\right] }\left( \mathbf{X}\right) -1.  \label{eqn_mainclaim}
\end{equation}%
The statement then follows from%
\begin{equation*}
C_{p,q}^{q}\left\vert h\right\vert _{q-var;\left[ 0,T\right] }^{q}\geq
C_{p,q}^{q}\sum_{i=0}^{\bar{N}_{\beta ,\left[ 0,T\right] }\left( \mathbf{X}%
\right) -1}\left\vert h\right\vert _{q-var;\left[ \tau _{i},\tau _{i+1}%
\right] }^{q}\geq \alpha ^{q/p}\bar{N}_{\beta ,\left[ 0,T\right] }\left( 
\mathbf{X}\right) .
\end{equation*}%
To show $\left( \ref{eqn_mainclaim}\right) $, we first notice that for every 
$i=0,\ldots ,\bar{N}_{\beta ,\left[ 0,T\right] }\left( \mathbf{X}\right) -1$,%
\begin{equation*}
\beta =|||\mathbf{X}\left( \omega \right) |||_{p-var;\left[ \tau _{i},\tau
_{i+1}\right] }^{p}=\sum_{k=1}^{\left[ p\right] }\left\Vert \mathbf{X}%
^{\left( k\right) }\left( \omega \right) \right\Vert _{p/k-var;\left[ \tau
_{i},\tau _{i+1}\right] }^{p/k}.
\end{equation*}%
Fix $i$. Then there is a $k\in \left\{ 1,\ldots ,\left[ p\right] \right\} $
such that $\left\Vert \mathbf{X}^{\left( k\right) }\left( \omega \right)
\right\Vert _{p/k-var;\left[ \tau _{i},\tau _{i+1}\right] }^{p/k}\geq \frac{%
\beta }{\left[ p\right] }$. Let $D=\left( t_{j}\right) _{j=0}^{M}$ be any
dissection of $\left[ \tau _{i},\tau _{i+1}\right] $. We define the vector%
\begin{equation*}
\mathbf{X}^{\left( k\right) }\left( \omega \right) :=\left( \mathbf{X}%
^{\left( k\right) }\left( \omega \right) _{t_{0},t_{1}},\ldots ,\mathbf{X}%
^{\left( k\right) }\left( \omega \right) _{t_{M-1},t_{M}}\right)
\end{equation*}%
and do the same for $\mathbf{X}^{\left( k\right) }\left( \omega -h\right) $
and for the mixed iterated integrals%
\begin{equation*}
\int_{\Delta ^{k}}dZ^{i_{1}}\otimes \ldots \otimes dZ^{i_{k}}\quad \text{%
where }Z^{i}=\left\{ 
\begin{array}{ccc}
X & \text{if} & i=0 \\ 
h & \text{if} & i=1%
\end{array}%
\right. .
\end{equation*}%
We have then%
\begin{equation*}
\mathbf{X}^{\left( k\right) }\left( \omega -h\right) =\sum_{\left(
i_{1},\ldots ,i_{k}\right) \in \left\{ 0,1\right\} ^{k}}\left( -1\right)
^{i_{1}+\ldots +i_{k}}\int_{\Delta ^{k}}dZ^{i_{1}}\otimes \ldots \otimes
dZ^{i_{k}}
\end{equation*}%
and by the triangle inequality,%
\begin{eqnarray}
&&\left\vert \int_{\Delta ^{k}}dh\otimes \ldots \otimes dh\right\vert
_{l^{p/k}}  \notag \\
&\geq &\left\vert \mathbf{X}^{\left( k\right) }\left( \omega \right)
\right\vert _{l^{p/k}}-\left( \left\vert \mathbf{X}^{\left( k\right) }\left(
\omega -h\right) \right\vert _{l^{p/k}}+\sum_{\substack{ \left( i_{1},\ldots
,i_{k}\right) \in \left\{ 0,1\right\} ^{k}  \\ 0<i_{1}+\ldots +i_{k}<k}}%
\left\vert \int_{\Delta ^{k}}dZ^{i_{1}}\otimes \ldots \otimes
dZ^{i_{k}}\right\vert _{l^{p/k}}\right) .  \label{eqn_rearranged_sum}
\end{eqnarray}%
Since $q<2$ and $p\geq q$, we can use Young and super-additivity of $%
\left\vert h\right\vert _{q-var}^{p}$ to see that%
\begin{eqnarray*}
\left\vert \int_{\Delta ^{k}}dh\otimes \ldots \otimes dh\right\vert
_{l^{p/k}}^{p/k} &\leq &c_{q,k}^{p/k}\sum_{j}\left\vert h\right\vert
_{q-var; \left[ t_{j},t_{j+1}\right] }^{p} \\
&\leq &c_{q,k}^{p/k}\left\vert h\right\vert _{q-var;\left[ \tau _{i},\tau
_{i+1}\right] }^{p}.
\end{eqnarray*}%
For the mixed integrals one has for any $u<v$%
\begin{equation*}
\left\vert \int_{\Delta _{u,v}^{k}}dZ^{i_{1}}\otimes \ldots \otimes
dZ^{i_{k}}\right\vert \leq c_{k,l,p,q}\left\vert h\right\vert _{q-var;\left[
u,v\right] }^{l}|||\mathbf{X}\left( \omega \right) |||_{p-var;\left[ u,v%
\right] }^{k-l}
\end{equation*}%
where $l=i_{1}+\ldots +i_{k}$ (this follows from Theorem 9.26 in \cite{FV10}%
). Hence we have, using H\"{o}lder's inequality and super-additivity%
\begin{eqnarray*}
\left\vert \int_{\Delta ^{k}}dZ^{i_{1}}\otimes \ldots \otimes
dZ^{i_{k}}\right\vert _{l^{p/k}}^{p/k} &\leq
&c_{k,l,p,q}^{p/k}\sum_{j}\left\vert h\right\vert _{q-var;\left[
t_{j},t_{j+1}\right] }^{\frac{lp}{k}}|||\mathbf{X}\left( \omega \right)
|||_{p-var;\left[ t_{j},t_{j+1}\right] }^{\frac{\left( k-l\right) p}{k}} \\
&\leq &c_{k,l,p,q}^{p/k}\left( \sum_{j}\left\vert h\right\vert _{q-var;\left[
t_{j},t_{j+1}\right] }^{p}\right) ^{l/k}\left( \sum_{j}|||\mathbf{X}\left(
\omega \right) |||_{p-var;\left[ t_{j},t_{j+1}\right] }^{p}\right) ^{\frac{%
k-l}{k}} \\
&\leq &c_{k,l,p,q}^{p/k}\left\vert h\right\vert _{q-var;\left[ \tau
_{i},\tau _{i+1}\right] }^{\frac{pl}{k}}|||\mathbf{X}\left( \omega \right)
|||_{p-var;\left[ \tau _{i},\tau _{i+1}\right] }^{\frac{p\left( k-l\right) }{%
k}}
\end{eqnarray*}%
and hence%
\begin{eqnarray*}
\left\vert \int_{\Delta ^{k}}dZ^{i_{1}}\otimes \ldots \otimes
dZ^{i_{k}}\right\vert _{l^{p/k}} &\leq &c_{k,p,q}\left\vert h\right\vert
_{q-var;\left[ \tau _{i},\tau _{i+1}\right] }^{l}|||\mathbf{X}\left( \omega
\right) |||_{p-var;\left[ \tau _{i},\tau _{i+1}\right] }^{k-l} \\
&=&c_{k,l,p,q}\left\vert h\right\vert _{q-var;\left[ \tau _{i},\tau _{i+1}%
\right] }^{l}\beta ^{\frac{k-l}{p}}.
\end{eqnarray*}%
By assumption,%
\begin{equation*}
\left\vert \mathbf{X}^{\left( k\right) }\left( \omega -h\right) \right\vert
_{l^{p/k}}\leq |||\mathbf{X}\left( \omega -h\right) |||_{p-var;\left[ 0,T%
\right] }^{k}\leq \alpha ^{k/p}.
\end{equation*}%
Plugging this into $\left( \ref{eqn_rearranged_sum}\right) $ yields%
\begin{equation*}
c_{q,k}\left\vert h\right\vert _{q-var;\left[ \tau _{i},\tau _{i+1}\right]
}^{k}\geq \left\vert \mathbf{X}^{\left( k\right) }\left( \omega \right)
\right\vert _{l^{p/k}}-\left( \alpha
^{k/p}+\sum_{l=1}^{k-1}c_{k,l,p,q}\left\vert h\right\vert _{q-var;\left[
\tau _{i},\tau _{i+1}\right] }^{l}\beta ^{\frac{k-l}{p}}\right) .
\end{equation*}%
Now we can take the supremum over all dissections $D$ and obtain, using $%
\left\Vert \mathbf{X}^{\left( k\right) }\left( \omega \right) \right\Vert
_{p/k-var;\left[ \tau _{i},\tau _{i+1}\right] }\geq \left( \frac{\beta }{%
\left[ p\right] }\right) ^{k/p}$,%
\begin{eqnarray*}
c_{q,k}\left\vert h\right\vert _{q-var;\left[ \tau _{i},\tau _{i+1}\right]
}^{k} &\geq &\left\Vert \mathbf{X}^{\left( k\right) }\left( \omega \right)
\right\Vert _{p/k-var;\left[ \tau _{i},\tau _{i+1}\right] }-\left( \alpha
^{k/p}+\sum_{l=1}^{k-1}c_{k,l,p,q}\left\vert h\right\vert _{q-var;\left[
\tau _{i},\tau _{i+1}\right] }^{l}\beta ^{\frac{k-l}{p}}\right) \\
&\geq &\left( \frac{\beta }{\left[ p\right] }\right) ^{k/p}-\left( \alpha
^{k/p}+\sum_{l=1}^{k-1}c_{k,l,p,q}\left\vert h\right\vert _{q-var;\left[
\tau _{i},\tau _{i+1}\right] }^{l}\beta ^{\frac{k-l}{p}}\right) \\
&=&\left( 2^{k}-1\right) \alpha ^{k/p}-\left( \sum_{l=1}^{k-1}\left( 2\left[
p\right] ^{1/p}\right) ^{k-l}c_{k,l,p,q}\left\vert h\right\vert _{q-var;%
\left[ \tau _{i},\tau _{i+1}\right] }^{l}\alpha ^{\frac{k-l}{p}}\right) .
\end{eqnarray*}%
By making constants larger if necessary, we may assume that there is a
constant $c_{k,p,q}$ such that%
\begin{equation*}
\left\vert h\right\vert _{q-var;\left[ \tau _{i},\tau _{i+1}\right]
}^{k}\geq \frac{\left( 2^{k}-1\right) }{c_{k,p,q}}\alpha ^{k/p}-\left(
\sum_{l=1}^{k-1}\left\vert h\right\vert _{q-var;\left[ \tau _{i},\tau _{i+1}%
\right] }^{l}\alpha ^{\frac{k-l}{p}}\right) .
\end{equation*}%
This implies that there is a constant $C_{k,p,q}$ depending on $c_{k,p,q}$
such that%
\begin{equation*}
C_{k}\left\vert h\right\vert _{q-var;\left[ \tau _{i},\tau _{i+1}\right]
}\geq \alpha ^{1/p}.
\end{equation*}%
Setting $C_{p,q}=\max \left\{ C_{1,p,q},\ldots ,C_{\left[ p\right]
,p,q}\right\} $ finally shows $\left( \ref{eqn_mainclaim}\right) $.
\end{proof}

Now we come to the main result.

\begin{corollary}
\label{cor_cll_results_single_gaussian}Let $X$ be a centred Gaussian process
in $\mathbb{R}^{d}$ with independent components and covariance $R_{X}$ of
finite $\rho $-variation, $\rho <2$. Consider the Gaussian $p$-rough paths $%
\mathbf{X}$ for $p>2\rho $ and assume that there is a continuous embedding%
\begin{equation*}
\iota \colon H\hookrightarrow C^{q-var}
\end{equation*}%
where $\frac{1}{p}+\frac{1}{q}>1$ and let $K\geq $ $\left\Vert \iota
\right\Vert _{op}$. Then for every~$\alpha >0$, $\bar{N}_{\alpha ,\left[ 0,T%
\right] }\left( \mathbf{X}\right) $ has a Weibull tail with shape $2/q$.
More precisely, there is a constant $C=C\left( p,q\right) $ such that%
\begin{equation*}
\mu \left\{ \bar{N}_{\alpha ,\left[ 0,T\right] }\left( \mathbf{X}\right)
>r\right\} \leq \exp \left\{ -\frac{1}{2}\left( \hat{a}+\frac{\alpha
^{1/p}r^{1/q}}{CK}\right) ^{2}\right\}
\end{equation*}%
for every $r>0$ where $\hat{a}>-\infty $ is chosen such that%
\begin{equation*}
\hat{a}\leq \Phi ^{-1}\mu \left\{ |||\mathbf{X}|||_{p-var;\left[ 0,T\right]
}^{p}\leq \frac{\alpha }{2^{p}\left[ p\right] }\right\} .
\end{equation*}
\end{corollary}

\begin{proof}
Set%
\begin{equation*}
A_{a}=\left\{ \omega :|||\mathbf{X}\left( \omega \right) |||_{p-var;\left[
0,T\right] }\leq a^{1/p}\right\} .
\end{equation*}%
Lemma \ref{lemma_Aa_pos_meas} guarantees that $A_{a}$ has positive measure
for any $a>0$. From Proposition \ref{prop_cll_theorem} we know that there is
a set $\tilde{E}$ of full measure such that whenever $\left\Vert \mathbf{X}%
\left( \omega -h\right) \right\Vert _{p-var;\left[ 0,T\right] }\leq a^{1/p}$
for $\omega \in \tilde{E}$, $h\in \mathcal{H}$ and $a>0$ we have%
\begin{equation*}
a^{1/p}\left( \bar{N}_{\beta ,\left[ 0,T\right] }\left( \mathbf{X}\right)
\right) ^{1/q}\leq c_{p,q}\left\vert h\right\vert _{q-var;\left[ 0,T\right]
}\leq c_{p,q}\left\Vert \iota \right\Vert _{op}\left\Vert h\right\Vert _{%
\mathcal{H}}
\end{equation*}%
where $\beta =2^{p}\left[ p\right] a$. Setting $a=\alpha /\left( 2^{p}\left[
p\right] \right) $, Corollary \ref{cor_conseq_borell} shows that%
\begin{equation*}
\mu \left( \left\{ \omega :\bar{N}_{\alpha ,\left[ 0,T\right] }\left( 
\mathbf{X}\left( \omega \right) \right) >r\right\} \right) \leq \exp \left(
-\left( \frac{\hat{a}}{\sqrt{2}}+\frac{\alpha ^{1/p}r^{1/q}}{2\sqrt{2}c_{p,q}%
\left[ p\right] ^{1/p}\left\Vert \iota \right\Vert _{op}}\right) ^{2}\right)
\end{equation*}%
where $\hat{a}\leq \Phi ^{-1}\mu \left( A_{a}\right) $.
\end{proof}

\begin{remark}
Corollary \ref{cor_cll_results_single_gaussian} remains valid if one
replaces $\bar{N}_{\alpha ,\left[ 0,T\right] }$ by $N_{\alpha ,\left[ 0,T%
\right] }$ in the statement. This follows directly from $\left( \ref%
{eqn_equiv_hom_rp_norms}\right) $ and Lemma \ref{lemma_transitivity_of_N} by
putting the constant $C$ of $\left( \ref{eqn_equiv_hom_rp_norms}\right) $ in
the constant $C_{p,q}$ of the respective corollaries.
\end{remark}

\begin{remark}
In \cite{FV10}, Proposition 15.7 is is shown that%
\begin{equation*}
\left\vert h\right\vert _{\rho -var}\leq \sqrt{V_{\rho -\text{var}}\left(
R_{X};\left[ 0,T\right] ^{2}\right) }\left\Vert h\right\Vert _{\mathcal{H}}
\end{equation*}%
holds for all $h\in \mathcal{H}$. Hence in the regime $\rho \in \lbrack
1,3/2)$ we can always choose $q=\rho $ and the conditions of Corollary \ref%
{cor_cll_results_single_gaussian} are fulfilled. For the fractional Brownian
motion with Hurst parameter $H$ one can show that $\rho =\frac{1}{2H}$ and $%
q>\frac{1}{H+1/2}$ are valid choices (cf. \cite{FV10}, chapter 15) and the
results of Corollary \ref{cor_cll_results_single_gaussian} remain valid
provided $H>1/4$.
\end{remark}

\begin{remark}
If $\mathbf{X}$ is a Gaussian rough paths, we know that $\left\Vert \mathbf{X%
}\right\Vert _{p-var}$ has a Gaussian tail (or a Weibull tail with shape
parameter 2), e.g. obtained by a non-linear Fernique Theorem, cf. \cite{FO10}%
, whereas Corollary \ref{cor_cll_results_single_gaussian} combined with
Lemma \ref{lemma_pvar_dom_by_N} only gives that $\left\Vert \mathbf{X}%
\right\Vert _{p-var}$ has a Weibull tail with shape $2/q$ and thus the
estimate is not sharp for $q>1$. On the other hand, Lemma \ref%
{lemma_pvar_dom_by_N} is robust and also available in situations where
Fernique- (or Borell-) type arguments are not directly available, e.g. in a
non-Gaussian setting.
\end{remark}

\section{Transitivity of the tail estimates under locally linear maps}

Existing maps for rough integrals and RDE solutions suggest that we consider
maps $\Psi $ such that $\left\Vert \Psi \left( \mathbf{x}\right) \right\Vert
_{p-var;I}\leq const.\left\Vert \mathbf{x}\right\Vert _{p-var;I}$ uniformly
over all intervals $I\subset \left[ 0,T\right] $ where $\left\Vert \mathbf{x}%
\right\Vert _{p-var;I}\leq R$, $R>0$. More formally,

\begin{definition}
\label{definition_loc_linear}We call $\Psi \colon C^{p-var}\left( \left[ 0,T%
\right] ;G^{N}\left( \mathbb{R}^{d}\right) \right) \rightarrow $ $%
C^{p-var}\left( \left[ 0,T\right] ;G^{M}\left( \mathbb{R}^{e}\right) \right) 
$ a \emph{locally linear map} if there is a $R\in (0,\infty ]$ such that%
\begin{equation}
\left\Vert \Psi \right\Vert _{R}:=\inf_{C>0}\left\{ \left\Vert \Psi \left( 
\mathbf{x}\right) \right\Vert _{p-var;[u,v]}\leq C\left\Vert \mathbf{x}%
\right\Vert _{p-var;[u,v]}\text{ for all }\left( u,v\right) \in \Delta ,%
\mathbf{x}\text{ s.t. }\left\Vert \mathbf{x}\right\Vert _{p-var;[u,v]}\leq R%
\text{ }\right\}  \notag
\end{equation}%
is finite.
\end{definition}

\begin{remark}
\begin{enumerate}
\item 

\item For $\lambda \in \mathbb{R}$, we denote by $\delta _{\lambda }$ the
dilation map. Set $\left( \delta _{\lambda }\Psi \right) \colon \mathbf{x}%
\mapsto \delta _{\lambda }\Psi \left( \mathbf{x}\right) $. Then $\left\Vert
\cdot \right\Vert _{R}$ is homogeneous w.r.t. dilation, e.g. $\left\Vert
\delta _{\lambda }\Psi \right\Vert _{R}=\left\vert \lambda \right\vert
\left\Vert \Psi \right\Vert _{R}$.

\item If $\Psi $ commutes with the dilation map $\delta $ modulo $p$%
-variation, e.g. $\left\Vert \Psi \left( \delta _{\lambda }\mathbf{x}\right)
\right\Vert _{p\text{-var;}I}=\left\Vert \delta _{\lambda }\Psi \left( 
\mathbf{x}\right) \right\Vert _{p\text{-var;}I}$ for any $\mathbf{x}$, $%
\lambda \in \mathbb{R}$ and intervall $I\subset \left[ 0,T\right] $, we have 
$\left\Vert \Psi \right\Vert _{R}=\left\Vert \Psi \right\Vert _{\infty }$
for any $R>0$. An example of such a map is the Lyons lift map%
\begin{equation*}
S_{N}\colon C^{p-var}\left( \left[ 0,T\right] ;G^{[p]}\left( \mathbb{R}%
^{d}\right) \right) \rightarrow C^{p-var}\left( \left[ 0,T\right]
;G^{N}\left( \mathbb{R}^{d}\right) \right)
\end{equation*}%
for which we have $\left\Vert S_{N}\right\Vert _{\infty }\leq C\left(
N,p\right) <\infty $, c.f. \cite{FV10}, Theorem 9.5.

\item If $\phi \colon \left[ 0,T\right] \rightarrow \phi \left[ 0,T\right]
\subset \left[ 0,T\right] $ is a bijective, continuous and increasing
function and $\mathbf{x}$ a rough path, we set $\mathbf{x}_{t}^{\phi }=$ $%
\mathbf{x}_{\phi \left( t\right) }$ and call $\mathbf{x}^{\phi }$ a
reparametrization of $\mathbf{x}$. If $\Psi $ commutes with
reparametrization modulo $p$-variation, e.g. $\left\Vert \Psi \left( \mathbf{%
x}^{\phi }\right) \right\Vert _{p-var;I}=\left\Vert \Psi \left( \mathbf{x}%
\right) ^{\phi }\right\Vert _{p-var;I}$ for any $\mathbf{x}$, $\phi $ and
intervall $I\subset \left[ 0,T\right] $, we have%
\begin{equation*}
\left\Vert \Psi \right\Vert _{R}:=\inf_{C>0}\left\{ \left\Vert \Psi \left( 
\mathbf{x}\right) \right\Vert _{p-var;[0,T]}\leq C\left\Vert \mathbf{x}%
\right\Vert _{p-var;[0,T]}\text{ for all }\mathbf{x}\text{ s.t. }\left\Vert 
\mathbf{x}\right\Vert _{p-var;[0,T]}\leq R\text{ }\right\} .
\end{equation*}%
This follows by a standard reparametrization argument. Examples of such maps
are rough integration over $1$-forms, e.g. $\mathbf{x\mapsto }%
\int_{0}^{\cdot }\varphi \left( x\right) \,d\mathbf{x}$, and the It\={o}%
-Lyons map, e.g. $\Psi \left( \mathbf{x}\right) _{s,t}=\mathbf{y}_{s,t}$
where $\mathbf{y}$ solves $d\mathbf{y}=V\left( \mathbf{x}\right) \,d\mathbf{x%
}$ with initial condition $\mathbf{y}_{0}\in G^{\left[ p\right] }\left( 
\mathbb{R}^{e}\right) $. In this case,%
\begin{equation*}
\left\Vert \Psi \right\Vert _{\infty }=\sup_{\mathbf{x}}\frac{\left\Vert
\Psi \left( \mathbf{x}\right) \right\Vert _{p-var;\left[ 0,T\right] }}{%
\left\Vert \mathbf{x}\right\Vert _{p-var;\left[ 0,T\right] }}
\end{equation*}%
(where $0/0:=0$) and we find the usual operator norm. (Note that, however,
we can not speak of linear maps in this context since rough paths spaces are
typically non-linear.)

\item Clearly, if $\left\Vert \Psi \right\Vert _{\infty }<\infty $, $%
\left\Vert \Psi \left( \mathbf{X}\right) \right\Vert _{p-var;\left[ s,t%
\right] }$ inherits the integrability properties of $\left\Vert \mathbf{X}%
\right\Vert _{p-var;\left[ s,t\right] }$. However, for the most interesting
maps, e.g. the It\={o}-Lyons map, we will \emph{not} have $\left\Vert \Psi
\right\Vert _{\infty }<\infty $, but $\left\Vert \Psi \right\Vert
_{R}<\infty $ for any finite $R>0$. In a way, the purpose of this section is
to show that one still has transitivity of integrability if one considers $%
N_{\alpha ,\left[ s,t\right] }\left( \mathbf{X}\right) $ instead of $%
\left\Vert \mathbf{X}\right\Vert _{p-var;\left[ s,t\right] }$.
\end{enumerate}
\end{remark}

\begin{lemma}
Let $\Psi $ and $\Phi $ be locally linear maps with $\left\Vert \Psi
\right\Vert _{R}<\infty $ and $\left\Vert \Phi \right\Vert _{R\left\Vert
\Psi \right\Vert _{R}}<\infty $. Then $\Psi \circ \Phi $ is again locally
linear and%
\begin{equation*}
\left\Vert \Phi \circ \Psi \right\Vert _{R}\leq \left\Vert \Phi \right\Vert
_{R\left\Vert \Psi \right\Vert _{R}}\left\Vert \Psi \right\Vert _{R}.
\end{equation*}
\end{lemma}

\begin{proof}
Let $\left\Vert \mathbf{x}\right\Vert _{p-var;[u,v]}\leq R$. Then $%
\left\Vert \Psi \left( \mathbf{x}\right) \right\Vert _{p-var;[u,v]}\leq
\left\Vert \Psi \right\Vert _{R}\left\Vert \mathbf{x}\right\Vert
_{p-var;[u,v]}\leq R\left\Vert \Psi \right\Vert _{R}$ which implies%
\begin{equation*}
\left\Vert \Phi \circ \Psi \left( \mathbf{x}\right) \right\Vert
_{p-var;[u,v]}\leq \left\Vert \Phi \right\Vert _{R\left\Vert \Psi
\right\Vert _{R}}\left\Vert \Psi \left( \mathbf{x}\right) \right\Vert
_{p-var;[u,v]}\leq \left\Vert \Phi \right\Vert _{R\left\Vert \Psi
\right\Vert _{R}}\left\Vert \Psi \right\Vert _{R}\left\Vert \mathbf{x}%
\right\Vert _{p-var;[u,v]}.
\end{equation*}
\end{proof}

The interesting property of locally linear maps is formulated in the next
proposition.

\begin{proposition}
\label{prop_main_deterministic}Let $\Psi \colon C^{p-var}\left( \left[ 0,T%
\right] ;G^{N}\left( \mathbb{R}^{d}\right) \right) \rightarrow $ $%
C^{p-var}\left( \left[ 0,T\right] ;G^{M}\left( \mathbb{R}^{e}\right) \right) 
$ be locally linear and $\left\Vert \Psi \right\Vert _{R}<\infty $ for some $%
R\in (0,\infty ]$. Then%
\begin{equation*}
N_{\alpha \left\Vert \Psi \right\Vert _{R}^{p},\left[ s,t\right] }\left(
\Psi \left( \mathbf{x}\right) \right) \leq N_{\alpha ,\left[ s,t\right]
}\left( \mathbf{x}\right)
\end{equation*}%
for any $s<t$ and $\alpha \in (0,R^{p}]$.
\end{proposition}

\begin{proof}
Follows directly from Lemma \ref{lemma_transitivity_of_N}.
\end{proof}

\subsection{Full RDEs}

Consider the full RDE%
\begin{equation}
d\mathbf{y}=V\left( \mathbf{y}\right) \,d\mathbf{x;\quad y}_{0}\in G^{\left[
p\right] }\left( \mathbb{R}^{e}\right)  \label{eqn_full_rde}
\end{equation}%
where $\mathbf{x}$ is a weak geometric $p$-rough path with values in $G^{%
\left[ p\right] }\left( \mathbb{R}^{d}\right) $, $V=\left( V_{i}\right)
_{i=1,\ldots .d}$ is a collection of $Lip^{\gamma }$-vector fields in $%
\mathbb{R}^{e}$ where $\gamma >p$ and $\mathbf{y}_{0}$ is the initial value.
Theorem 10.36 and 10.38 in \cite{FV10} state that $\left( \ref{eqn_full_rde}%
\right) $ possesses a unique solution $\mathbf{y}$ which is a weak geometric 
$p$-rough path with values in $G^{\left[ p\right] }\left( \mathbb{R}%
^{e}\right) $.

\begin{corollary}
\label{cor_full_rdes}The It\={o}-Lyons map $\Psi \colon \mathbf{x\mapsto y}$
is locally linear with%
\begin{equation}
\left\Vert \Psi \right\Vert _{R}\leq K\left( \left\Vert V\right\Vert _{\text{%
Lip}^{\gamma -1}}\vee \left\Vert V\right\Vert _{\text{Lip}^{\gamma
-1}}^{p}R^{p-1}\right)  \label{eqn_ito_lyons_map_loc_lin}
\end{equation}%
for any $R\in (0,\infty )$ where $K$ only depends on $p$ and $\gamma $.
Moreover, if $\left\Vert V\right\Vert _{\text{Lip}^{\gamma -1}}\leq \nu $,
then for any $\alpha >0$ there is a constant $C=C\left( p,\gamma ,\nu
,\alpha \right) $ such that%
\begin{equation*}
N_{\alpha ,\left[ s,t\right] }\left( \mathbf{y}\right) \leq C\left(
N_{\alpha ,\left[ s,t\right] }\left( \mathbf{x}\right) +1\right)
\end{equation*}
for any \thinspace $s<t$.
\end{corollary}

\begin{proof}
$\left( \ref{eqn_ito_lyons_map_loc_lin}\right) $ follows from the estimate
(10.26) of Theorem 10.36 in \cite{FV10}. From Proposition \ref%
{prop_main_deterministic} we obtain%
\begin{equation*}
N_{\beta ,\left[ s,t\right] }\left( \Psi \left( \mathbf{x}\right) \right)
\leq N_{\alpha ,\left[ s,t\right] }\left( \mathbf{x}\right)
\end{equation*}

where $\beta =\alpha \left\Vert \Psi \right\Vert _{\alpha ^{1/p}}^{p}$. This
already shows the claim if $\left\Vert \Psi \right\Vert _{\alpha
^{1/p}}^{p}\leq 1$. In the case $\left\Vert \Psi \right\Vert _{\alpha
^{1/p}}^{p}>1$, we conclude with Lemma \ref{lemma_Nalpha_Nbeta_estimate}.
\end{proof}

\subsection{Rough integrals}

If $\mathbf{x}$ is a $p$-rough path and $\varphi =\left( \varphi _{i}\right)
_{i=1,\ldots ,d}$ a collection of $Lip^{\gamma -1}\left( \mathbb{R}^{d},%
\mathbb{R}^{e}\right) $-maps, one can define the rough integral%
\begin{equation}
\int \varphi \left( x\right) \,d\mathbf{x}  \label{eqn_rough_integral}
\end{equation}%
as an element in $C^{p-var}\left( \left[ 0,T\right] ;G^{\left[ p\right]
}\left( \mathbb{R}^{e}\right) \right) $ (c.f. \cite{FV10}, chapter 10.6).

\begin{corollary}
\label{cor_rough_integrals}The map $\Psi \colon \mathbf{x\mapsto z}$, $%
\mathbf{z}$ given by the rough integral $\left( \ref{eqn_rough_integral}%
\right) $, is locally linear with%
\begin{equation}
\left\Vert \Psi \right\Vert _{R}\leq K\left\Vert \varphi \right\Vert _{\text{%
Lip}^{\gamma -1}}\left( 1\vee R^{p-1}\right)
\label{eqn_rough_integr_loc_lin}
\end{equation}
for any $R\in (0,\infty )$ where $K$ only depends on $p$ and $\gamma $.
Moreover, if $\left\Vert \varphi \right\Vert _{\text{Lip}^{\gamma -1}}\leq
\nu $, then for any $\alpha >0$ there is a constant \thinspace $C=C\left(
p,\gamma ,\nu ,\alpha \right) $ such that%
\begin{equation*}
N_{\alpha ,\left[ s,t\right] }\left( \mathbf{z}\right) \leq C\left(
N_{\alpha ,\left[ s,t\right] }\left( \mathbf{x}\right) +1\right)
\end{equation*}%
for any $s<t$.
\end{corollary}

\begin{proof}
$\left( \ref{eqn_rough_integr_loc_lin}\right) $ follows from \cite{FV10},
Theorem 10.47. One proceeds as in the proof of Corollary \ref{cor_full_rdes}.
\end{proof}

\section{Linear RDEs}

For a $p$-rough path $\mathbf{x}$, consider the full linear RDE%
\begin{equation}
d\mathbf{y}=V\left( \mathbf{y}\right) \,d\mathbf{x;\quad y}_{0}\in G^{\left[
p\right] }\left( \mathbb{R}^{e}\right)  \label{eqn_linear_rde}
\end{equation}%
where $V=\left( V_{i}\right) _{i=1,\ldots ,d}$ is a collection of linear
vector fields of the form $V_{i}\left( z\right) =A_{i}z+b_{i}$, $A_{i}$ are $%
e\times e$ matrices and $b_{i}\in \mathbb{R}^{e}$. It is well-known (e.g. 
\cite{FV10}, section 10.7) that in this case $\left( \ref{eqn_linear_rde}%
\right) $ has a unique solution $\mathbf{y}$. Unfortunately, the map $\Psi
\colon \mathbf{x\mapsto y}$ is not locally linear in the sense of Definition %
\ref{definition_loc_linear} and our tools of the former section do not
apply. However, we can do a more direct analysis and obtain a different
transitivity of the tail estimates.

Let $\nu $ be a bound on $\max_{i}\left( |A_{i}|+|b_{i}|\right) $ and set $%
y=\pi _{1}\left( \mathbf{y}\right) $. In \cite{FV10}, Theorem 10.53 one sees
that there is a constant $C$ depending only on $p$ such that%
\begin{equation}
\left\Vert \mathbf{y}_{s,t}\right\Vert \leq C\left( 1+\left\vert
y_{s}\right\vert \right) \nu \left\Vert \mathbf{x}\right\Vert _{p-var;\left[
s,t\right] }\exp \left( C\nu ^{p}\left\Vert \mathbf{x}\right\Vert _{p-var;%
\left[ s,t\right] }^{p}\right)  \label{eqn_key_estimate_linear_rdes}
\end{equation}%
holds for all $s<t\in \lbrack 0,T]$. (Strictly speaking, we only find the
estimate for $\left( s,t\right) =\left( 0,1\right) $, the general case
follows by reparametrization.) We start with an estimate for the supremum
norm of $y$.

\begin{lemma}
\label{lemma_sup_norm_linear_rde}\bigskip For any $\alpha >0$ there is a
constant $C=C\left( p,\nu ,\alpha \right) $ such that%
\begin{equation*}
\left\vert y\right\vert _{\infty ;\left[ s,t\right] }\leq C\left(
1+\left\vert y_{s}\right\vert \right) \exp \left( CN_{\alpha ;\left[ s,t%
\right] }\left( \mathbf{x}\right) \right)
\end{equation*}%
holds for any \thinspace $s<t$.
\end{lemma}

\begin{proof}
From $\left( \ref{eqn_key_estimate_linear_rdes}\right) $ we have%
\begin{equation}
\left\vert y_{u,v}\right\vert \leq C\left( 1+\left\vert y_{u}\right\vert
\right) \nu \left\Vert \mathbf{x}\right\Vert _{p-var;\left[ u,v\right] }\exp
\left( C\nu ^{p}\left\Vert \mathbf{x}\right\Vert _{p-var;\left[ u,v\right]
}^{p}\right)  \label{eqn_reparam_linear_rde}
\end{equation}%
for any $u<v\in \left[ s,t\right] $. From $\left\vert y_{u,v}\right\vert
=\left\vert y_{s,v}-y_{s,u}\right\vert \geq \left\vert y_{s,v}\right\vert
-\left\vert y_{s,u}\right\vert $ we obtain%
\begin{eqnarray*}
\left\vert y_{s,v}\right\vert &\leq &C\left( 1+\left\vert y_{u}\right\vert
\right) \nu \left\Vert \mathbf{x}\right\Vert _{p-var;\left[ u,v\right] }\exp
\left( C\nu ^{p}\left\Vert \mathbf{x}\right\Vert _{p-var;\left[ u,v\right]
}^{p}\right) +\left\vert y_{s,u}\right\vert \\
&\leq &C\left( 1+\left\vert y_{s}\right\vert +\left\vert y_{s,u}\right\vert
\right) \exp \left\{ C\nu ^{p}\left\Vert \mathbf{x}\right\Vert _{p-var;\left[
u,v\right] }^{p}\right\}
\end{eqnarray*}%
by making $C$ larger. Now let $s=\tau _{0}<\ldots <\tau _{N}<\tau
_{M+1}=u\leq t$ with $M\geq 0$. By induction, one sees that%
\begin{eqnarray*}
\left\vert y_{s,u}\right\vert &\leq &C^{M+1}\left( \left( M+1\right) \left(
1+\left\vert y_{s}\right\vert \right) \right) \exp \left\{
C\sum_{i=0}^{M}\nu ^{p}\left\Vert \mathbf{x}\right\Vert _{p-var;\left[ \tau
_{i},\tau _{i+1}\right] }^{p}\right\} \\
&\leq &C^{M+1}\left( 1+\left\vert y_{s}\right\vert \right) \exp \left\{
C\sum_{i=0}^{M}\nu ^{p}\left\Vert \mathbf{x}\right\Vert _{p-var;\left[ \tau
_{i},\tau _{i+1}\right] }^{p}\right\} .
\end{eqnarray*}%
This shows that for every $u\in \left[ s,t\right] $,%
\begin{eqnarray*}
\left\vert y_{s,u}\right\vert &\leq &C^{\left( N_{\alpha ;\left[ s,t\right]
}\left( \mathbf{x}\right) +1\right) }\left( 1+\left\vert y_{s}\right\vert
\right) \exp \left( C\nu ^{p}\alpha \left( N_{\alpha ;\left[ s,t\right]
}\left( \mathbf{x}\right) +1\right) \right) \\
&=&\left( 1+\left\vert y_{s}\right\vert \right) \exp \left\{ \left( \log
\left( C\right) +C\nu ^{p}\alpha \right) \left( N_{\alpha ;\left[ s,t\right]
}\left( \mathbf{x}\right) +1\right) \right\}
\end{eqnarray*}%
and hence%
\begin{equation*}
\sup_{u\in \left[ s,t\right] }\left\vert y_{s,u}\right\vert \leq C\left(
1+\left\vert y_{s}\right\vert \right) \exp \left( CN_{\alpha ;\left[ s,t%
\right] }\left( \mathbf{x}\right) \right)
\end{equation*}%
for a constant $C=C\left( p,\nu ,\alpha \right) $ and therefore also%
\begin{equation*}
\left\vert y\right\vert _{\infty ;\left[ s,t\right] }\leq C\left(
1+\left\vert y_{s}\right\vert \right) \exp \left( CN_{\alpha ;\left[ s,t%
\right] }\left( \mathbf{x}\right) \right) .
\end{equation*}%
\bigskip
\end{proof}

\begin{corollary}
\label{cor_linear_RDEs}Let $\alpha >0$. Then there is a constant $C=C\left(
p,\nu ,\alpha \right) $ such that%
\begin{equation*}
N_{\alpha ,\left[ s,t\right] }\left( \mathbf{y}\right) \leq C\left(
1+\left\vert y_{s}\right\vert \right) ^{p}\exp \left( CN_{\alpha ;\left[ s,t%
\right] }\left( \mathbf{x}\right) \right)
\end{equation*}%
for any \thinspace $s<t$.
\end{corollary}

\begin{proof}
Using $\left( \ref{eqn_key_estimate_linear_rdes}\right) $ we can deduce that%
\begin{equation*}
\left\Vert \mathbf{y}_{u,v}\right\Vert \leq C\left( 1+\left\vert
y\right\vert _{\infty ;\left[ s,t\right] }\right) \nu \left\Vert \mathbf{x}%
\right\Vert _{p-var;\left[ u,v\right] }\exp \left( C\nu ^{p}\left\Vert 
\mathbf{x}\right\Vert _{p-var;\left[ u,v\right] }^{p}\right)
\end{equation*}%
holds for any $u<v\in \left[ s,t\right] $ and hence also%
\begin{equation*}
\left\Vert \mathbf{y}\right\Vert _{p-var;\left[ u,v\right] }\leq C\left(
1+\left\vert y\right\vert _{\infty ;\left[ s,t\right] }\right) \nu
\left\Vert \mathbf{x}\right\Vert _{p-var;\left[ u,v\right] }\exp \left( C\nu
^{p}\left\Vert \mathbf{x}\right\Vert _{p-var;\left[ u,v\right] }^{p}\right)
\end{equation*}%
for any $u<v\in \left[ s,t\right] $. Now take $u<v\in \left[ s,t\right] $
such that $\left\Vert \mathbf{x}\right\Vert _{p-var;\left[ u,v\right]
}^{p}\leq \alpha $. We then have%
\begin{equation*}
\left\Vert \mathbf{y}\right\Vert _{p-var;\left[ u,v\right] }^{p}\leq \tilde{C%
}\left\Vert \mathbf{x}\right\Vert _{p-var;\left[ u,v\right] }^{p}
\end{equation*}%
where%
\begin{equation*}
\tilde{C}=C^{p}\left( 1+\left\vert y\right\vert _{\infty ;\left[ s,t\right]
}\right) ^{p}\nu ^{p}\exp \left( pC\nu ^{p}\alpha \right) \text{.}
\end{equation*}%
From Lemma \ref{lemma_transitivity_of_N},%
\begin{equation*}
N_{\tilde{C}\alpha ,\left[ s,t\right] }\left( \mathbf{y}\right) \leq
N_{\alpha ,\left[ s,t\right] }\left( \mathbf{x}\right) .
\end{equation*}

If $\tilde{C}\leq 1$, this already shows the claim. For $\tilde{C}>1$, we
use Lemma \ref{lemma_Nalpha_Nbeta_estimate} and Lemma \ref%
{lemma_sup_norm_linear_rde} to see that%
\begin{eqnarray*}
N_{\alpha ,\left[ s,t\right] }\left( \mathbf{y}\right) &\leq &\left( 2N_{%
\tilde{C}\alpha ,\left[ s,t\right] }\left( \mathbf{y}\right) +1\right) 
\tilde{C} \\
&\leq &C\left( N_{\alpha ,\left[ s,t\right] }\left( \mathbf{x}\right)
+1\right) \left( 1+\left\vert y\right\vert _{\infty ;\left[ s,t\right]
}\right) ^{p} \\
&\leq &C\left( 1+\left\vert y_{s}\right\vert \right) ^{p}\exp \left(
CN_{\alpha ,\left[ s,t\right] }\left( \mathbf{x}\right) \right) .
\end{eqnarray*}
\end{proof}

\begin{remark}[Unbounded vector fields]
\label{remark_unbounded_vf}Let $\mathbf{x}$ be a $p$-rough path. Consider a
collection $V=\left( V_{i}\right) _{1\leq i\leq d}$ of locally $Lip^{\gamma
-1}$-vector fields on $\mathbb{R}^{e}$, $\gamma \in \left( p,[p]+1\right) $,
such that $V_{i}$ are Lipschitz continuous and the vector fields $%
V^{[p]}=\left( V_{i_{1}},\ldots ,V_{i_{[p]}}\right) _{i_{1},\ldots
,i_{[p]}\in \{1,\ldots ,d\}}$ are $(\gamma -[p])$-H\"{o}lder continuous.
Then the RDE%
\begin{equation*}
d\mathbf{y}=V\left( \mathbf{y}\right) \,d\mathbf{x;\quad y}_{0}\in G^{\left[
p\right] }\left( \mathbb{R}^{e}\right)
\end{equation*}%
has a unique solution (c.f. \cite{FV10}, Exercise 10.56 and the solution
thereafter and \cite{Lej09}). Moreover, in \cite{FV10} it is shown that%
\begin{equation*}
\left\Vert \mathbf{y}_{0,1}\right\Vert \leq C\left( 1+\left\vert
y_{0}\right\vert \right) \nu \left\Vert \mathbf{x}\right\Vert _{p-var;\left[
0,1\right] }\exp \left( C\nu ^{p}\left\Vert \mathbf{x}\right\Vert _{p-var;%
\left[ 0,1\right] }^{p}\right)
\end{equation*}%
where $C=C\left( p,\gamma \right) $ and $\nu $ is a bound on $\left\vert
V^{[p]}\right\vert _{\left( \gamma -[p]\right) \text{-H\"{o}l}}^{1/[p]}\vee
\sup_{y,z}\frac{\left\vert V\left( y\right) -V\left( z\right) \right\vert }{%
\left\vert y-z\right\vert }$. This shows that Lemma \ref%
{lemma_sup_norm_linear_rde} and Corollary \ref{cor_linear_RDEs} apply for $%
\mathbf{y}$, hence for any $\alpha >0$ there is a constant $C=C\left(
p,\gamma ,\nu ,\alpha \right) $ such that%
\begin{equation*}
N_{\alpha ,\left[ s,t\right] }\left( \mathbf{y}\right) \leq C\left(
1+\left\vert y_{s}\right\vert \right) ^{p}\exp \left( CN_{\alpha ;\left[ s,t%
\right] }\left( \mathbf{x}\right) \right)
\end{equation*}%
for all $s<t$ in this case.
\end{remark}

\section{Applications in stochastic analysis}

\subsection{Tail estimates for stochastic integrals and solutions of SDEs
driven by Gaussian signals}

We now apply our results to solutions of SDEs and stochastic integrals
driven by Gaussian signals, i.e. a Gaussian rough path $\mathbf{X}$. Remark
that all results here may be immediately formulated for SDEs and stochastic
integrals driven by random rough paths as along as suitable quantitative
Weibull-tail estimate for $N_{\alpha ,\left[ 0,T\right] }\left( \mathbf{X}%
\right) $ are assumed.

We first consider the non-linear case:

\begin{proposition}
\label{prop_tail_estimates_sdes_gaussian}Let $X$ be a centred Gaussian
process in $\mathbb{R}^{d}$ with independent components and covariance $%
R_{X} $ of finite $\rho $-variation, $\rho <2$. Consider the Gaussian $p$%
-rough paths $\mathbf{X}$ for $p>2\rho $ and assume that there is a
continuous embedding%
\begin{equation*}
\iota \colon H\hookrightarrow C^{q-var}
\end{equation*}%
where $\frac{1}{p}+\frac{1}{q}>1$. Let $Y\colon \lbrack 0,T]\rightarrow 
\mathbb{R}^{e}$ be the pathwise solution of the stochastic RDE%
\begin{equation*}
dY=V\left( Y\right) \,d\mathbf{X;\quad }Y_{0}\in \mathbb{R}^{e}
\end{equation*}%
where $V=\left( V_{i}\right) _{i=1,\ldots .d}$ is a collection of $%
Lip^{\gamma }$-vector fields in $\mathbb{R}^{e}$ with $\gamma >p$. Moreover,
let $Z\colon \lbrack 0,T]\rightarrow \mathbb{R}^{e}$ be the stochastic
integral given by%
\begin{equation*}
Z_{t}=\pi _{1}\left( \int_{0}^{t}\varphi \left( X\right) \,d\mathbf{X}\right)
\end{equation*}%
where $\varphi =\left( \varphi _{i}\right) _{i=1,\ldots ,d}$ is a collection
of $Lip^{\gamma -1}\left( \mathbb{R}^{d},\mathbb{R}^{e}\right) $-maps, $%
\gamma >p$. Then both $\left\Vert Y\right\Vert _{p-var;[0,T]}$ and $%
\left\Vert Z\right\Vert _{p-var;[0,T]}$ have Weibull tails with shape
parameter $2/q$. More precisely, if $K\geq \left\Vert \iota \right\Vert
_{op} $, $M\geq V_{\rho -\text{var}}\left( R;\left[ 0,T\right] ^{2}\right) $
and $\nu \geq \left\Vert V\right\Vert _{\text{Lip}^{\gamma -1}}$ there is a
constant $\eta =\eta \left( p,q,\rho ,\gamma ,\nu ,K,M\right) >0$ such that%
\begin{equation*}
P\left( \left\Vert Y\right\Vert _{p-var;[0,T]}>r\right) \leq \frac{1}{\eta }%
\exp \left( -\eta r^{2/q}\right) \quad \text{for all }r\geq 0
\end{equation*}%
and the same holds for $\left\Vert Z\right\Vert _{p-var;[0,T]}$ if $\nu \geq
\left\Vert \varphi \right\Vert _{\text{Lip}^{\gamma -1}}$ instead. In
particular, $\left\Vert Y\right\Vert _{p-var;[0,T]}$ and $\left\Vert
Z\right\Vert _{p-var;[0,T]}$ have finite exponential moments as long as $q<2$%
.
\end{proposition}

\begin{proof}
From \ref{lemma_Aa_pos_meas} we know that there is a $\alpha =\alpha \left(
\rho ,p,M\right) $ such that%
\begin{equation*}
P\left\{ |||\mathbf{X}|||_{p-var;\left[ 0,T\right] }^{p}\leq \frac{\alpha }{%
2^{p}\left[ p\right] }\right\} \geq \frac{1}{2}.
\end{equation*}%
Hence, by Corollary \ref{cor_cll_results_single_gaussian}, applied with $%
\hat{a}=\Phi ^{-1}\left( \frac{1}{2}\right) =0$, and the remark thereafter,%
\begin{equation*}
P\left\{ N_{\alpha ,\left[ 0,T\right] }\left( \mathbf{X}\right) >r\right\}
\leq \exp \left\{ -\frac{1}{2}\left( \frac{\alpha ^{1/p}r^{1/q}}{c_{1}}%
\right) ^{2}\right\} \quad \text{for all }r\geq 0
\end{equation*}%
with $c_{1}=c_{1}\left( p,q,K,M\right) $. Corollary \ref{cor_full_rdes}
shows that there is a constant $c_{2}=c_{2}\left( p,q,K,M,\gamma ,\nu
\right) $ such that also%
\begin{equation*}
P\left\{ N_{\alpha ,\left[ 0,T\right] }\left( \mathbf{Y}\right) >r\right\}
\leq c_{2}\exp \left\{ -\frac{r^{2/q}}{c_{2}}\right\} \quad \text{for all }%
r\geq 0.
\end{equation*}

From Lemma \ref{lemma_pvar_dom_by_N} we see that%
\begin{equation*}
\left\Vert Y\right\Vert _{p-var;[0,T]}\leq \left\Vert \mathbf{Y}\right\Vert
_{p-var;[0,T]}\leq \alpha ^{1/p}\left( N_{\alpha ,\left[ 0,T\right] }\left( 
\mathbf{Y}\right) +1\right)
\end{equation*}

which shows the claim for $\left\Vert Y\right\Vert _{p-var;[0,T]}$. The same
holds true for $\left\Vert Z\right\Vert _{p-var;[0,T]}$ by using Corollary %
\ref{cor_rough_integrals}.
\end{proof}

\begin{remark}
In the Brownian motion case ($q=1$), we recover the well-known fact that
solutions $Y$ of the Stratonovich SDE%
\begin{equation*}
dY=V\left( Y\right) \,\circ dB;\quad Y_{0}\in \mathbb{R}^{e}
\end{equation*}%
have Gaussian tails at any fixed point $Y_{t}$ provided $V$ is sufficiently
smooth. We also recover that the Stratonovich integral%
\begin{equation*}
\int_{0}^{t}\varphi \left( B\right) \,\circ dB
\end{equation*}%
has finite Gaussian tails for every $t\geq 0$, $\varphi $ sufficiently
smooth.
\end{remark}

\begin{proposition}
\label{prop_tail_estimates_linear_sdes_gaussian}Let $X$ be as in Proposition %
\ref{prop_tail_estimates_sdes_gaussian}. Let $Y\colon \lbrack
0,T]\rightarrow \mathbb{R}^{e}$ be the pathwise solution of the stochastic
linear RDE%
\begin{equation*}
dY=V\left( Y\right) \,d\mathbf{X;\quad }Y_{0}\in \mathbb{R}^{e}
\end{equation*}%
where $V=\left( V_{i}\right) _{i=1,\ldots ,d}$ is a collection of linear
vector fields of the form $V_{i}\left( z\right) =A_{i}z+b_{i}$, $A_{i}$ are $%
e\times e$ matrices and $b_{i}\in \mathbb{R}^{e}$. Then $\log \left(
\left\Vert Y\right\Vert _{p-var;[0,T]}\right) $ has a Weibull tail with
shape $2/q$. More precisely, if $K\geq \left\Vert \iota \right\Vert _{op}$, $%
M\geq V_{\rho -\text{var}}\left( R;\left[ 0,T\right] ^{2}\right) $ and $\nu
\geq \max_{i}\left( |A_{i}|+|b_{i}|\right) $ there is a constant $\eta =\eta
\left( p,q,\rho ,\nu ,K,M\right) >0$ such that%
\begin{equation*}
P\left( \log \left( \left\Vert Y\right\Vert _{p-var;[0,T]}\right) >r\right)
\leq \frac{1}{\eta }\exp \left( -\eta r^{2/q}\right) \quad \text{for all }%
r\geq 0
\end{equation*}
In particular, $\left\Vert Y\right\Vert _{p-var;[0,T]}$ has finite $L^{s}$%
-moments for any $s>0$ provided $q<2$.
\end{proposition}

\begin{proof}
Same as for Proposition \ref{prop_tail_estimates_sdes_gaussian} using
Corollary \ref{cor_linear_RDEs}.
\end{proof}

\begin{remark}
In the case $q=1$, which covers Brownian driving signals, we have log-normal
tails. This is in agreement with trivial examples such as the standard
Black-Scholes model in which the stock price $S_{t}$ is log-normally
distributed.
\end{remark}

\begin{remark}
The same conclusion holds for unbounded vector fields as seen in remark \ref%
{remark_unbounded_vf}.
\end{remark}

\subsection{The Jacobian of the solution flow for SDEs driven by Gaussian
signals}

Let $x\colon \left[ 0,T\right] \rightarrow \mathbb{R}^{d}$ be smooth and let 
$V=\left( V^{1},\ldots ,V^{d}\right) \colon \mathbb{R}^{e}\rightarrow 
\mathbb{R}^{e}$ be a collection of vector fields. We can interpret $V$ as a
function $V\colon \mathbb{R}^{e}\rightarrow L\left( \mathbb{R}^{d},\mathbb{R}%
^{e}\right) $ with derivative $DV\colon \mathbb{R}^{e}\rightarrow L\left( 
\mathbb{R}^{e},L\left( \mathbb{R}^{d},\mathbb{R}^{e}\right) \right) \cong
L\left( \mathbb{R}^{d},\text{End}\left( \mathbb{R}^{e}\right) \right) $. It
is well-known that for sufficiently smooth $V$, the ODE%
\begin{equation*}
dy=V\left( y\right) \,dx
\end{equation*}%
has a solution for every starting point $y_{0}$ and the solution flow $%
y_{0}\rightarrow U_{t\leftarrow 0}\left( y_{0}\right) =y_{t}$ is (Fr\'{e}%
chet) differentiable. We denote its derivative by $J_{t\leftarrow
0}^{x}\left( y_{0}\right) =DU_{t\leftarrow 0}\left( \cdot \right) |_{\cdot
=y_{0}}$. Moreover, for fixed $y_{0}$, the Jacobian $J_{t}=J_{t\leftarrow
0}^{x}\left( y_{0}\right) $ is given as the solution of the linear ODE%
\begin{equation*}
dJ_{t}=dM_{t}\cdot J_{t};\quad J_{0}=Id
\end{equation*}%
where $M_{t}\in \,$End$\left( \mathbb{R}^{e}\right) $ is given by the
integral%
\begin{equation}
M_{t}=\int_{0}^{t}DV\left( y_{s}\right) \,dx_{s}\text{.}
\label{eqn_def_M_for_jacobian}
\end{equation}%
If $\mathbf{x}$ is a $p$-rough path, one proceeds in a similar fashion.
First, in order to make sense of $\left( \ref{eqn_def_M_for_jacobian}\right) 
$ if $\mathbf{x}$ and $\mathbf{y}$ are rough paths, one has to define the
joint rough path $\left( \mathbf{x,y}\right) =\mathbf{z}\in C^{p-var}\left( %
\left[ 0,T\right] ,G^{\left[ p\right] }\left( \mathbb{R}^{d}\oplus \mathbb{R}%
^{e}\right) \right) $ first. To do so, one defines $\mathbf{z}$ as the
solution of the full RDE%
\begin{equation*}
d\mathbf{z}=\tilde{V}\left( \mathbf{z}\right) \,d\mathbf{x;\quad z}_{0}=\exp
\left( 0,y_{0}\right) .
\end{equation*}%
where $\tilde{V}=\left( Id,V\right) $. Then, one defines $\mathbf{M}\in
C^{p-var}\left( \left[ 0,T\right] ,G^{\left[ p\right] }\left( \mathbb{R}%
^{e\times e}\right) \right) $ as the rough integral%
\begin{equation*}
\mathbf{M}_{t}=\int_{0}^{t}\phi \left( z\right) \,d\mathbf{z}
\end{equation*}%
where $\phi \colon \mathbb{R}^{d}\oplus \mathbb{R}^{e}\rightarrow L\left( 
\mathbb{R}^{d}\oplus \mathbb{R}^{e},\text{End}\left( \mathbb{R}^{e}\right)
\right) $ is given by $\phi \left( x,y\right) \left( x^{\prime },y^{\prime
}\right) =DV\left( y\right) \left( x^{\prime }\right) $ for all $x,x^{\prime
}\in \mathbb{R}^{d}$ and $y,y^{\prime }\in \mathbb{R}^{e}$. Finally, one
obtains $J_{t}^{\mathbf{x}}=J_{t\leftarrow 0}^{\mathbf{x}}\left(
y_{0}\right) $ as the solution of the linear RDE%
\begin{equation*}
dJ_{t}^{\mathbf{x}}=d\mathbf{M}_{t}\cdot J_{t}^{\mathbf{x}};\quad J_{0}=Id.
\end{equation*}%
All this can be made rigorous; for instance, see \cite{FV10}, Theorem 11.3.
Next, we give an alternative proof of the main result of \cite{CLL},
slightly sharpened in the sense that we consider the $p$-variation norm
instead of the supremum norm.

\begin{proposition}
Let $X$ be a centred Gaussian process in $\mathbb{R}^{d}$ with independent
components and covariance $R_{X}$ of finite $\rho $-variation, $\rho <2$.
Consider the Gaussian $p$-rough paths $\mathbf{X}$ for $p>2\rho $ and assume
that there is a continuous embedding%
\begin{equation*}
\iota \colon H\hookrightarrow C^{q-var}
\end{equation*}%
where $\frac{1}{p}+\frac{1}{q}>1$. Then $\log \left( \left\Vert J_{\cdot
\leftarrow 0}^{\mathbf{X}}\left( y_{0}\right) \right\Vert _{p-var;\left[ 0,T%
\right] }\right) $ has a Weibull tail with shape $2/q$. In particular, if $%
q<2$, this implies that $\left\Vert J_{\cdot \leftarrow 0}^{\mathbf{X}%
}\left( y_{0}\right) \right\Vert _{p-var;\left[ 0,T\right] }$ has finite $%
L^{r}$-moments for any $r>0$.
\end{proposition}

\begin{proof}
From Corollary \ref{cor_cll_results_single_gaussian} we know that $N_{1,%
\left[ 0,T\right] }\left( \mathbf{X}\right) $ has a Weibull tail with shape $%
2/q$. Combining the Corollaries \ref{cor_full_rdes}, \ref%
{cor_rough_integrals} and \ref{cor_linear_RDEs} shows that there is a
constant $C$ such that%
\begin{equation*}
\log \left( N_{1,\left[ 0,T\right] }\left( \mathbf{J}_{\cdot \leftarrow 0}^{%
\mathbf{X}}\left( y_{0}\right) \right) +1\right) \leq C\left( N_{1;\left[ 0,T%
\right] }\left( \mathbf{X}\right) +1\right) .
\end{equation*}%
From Lemma \ref{lemma_pvar_dom_by_N} we know that%
\begin{equation*}
\left\Vert J_{\cdot \leftarrow 0}^{\mathbf{X}}\left( y_{0}\right)
\right\Vert _{p-var;\left[ 0,T\right] }\leq \left\Vert \mathbf{J}_{\cdot
\leftarrow 0}^{\mathbf{X}}\left( y_{0}\right) \right\Vert _{p-var;\left[ 0,T%
\right] }\leq N_{1,\left[ 0,T\right] }\left( \mathbf{J}_{\cdot \leftarrow
0}^{\mathbf{X}}\left( y_{0}\right) \right) +1.
\end{equation*}
\end{proof}

\subsection{An example from rough SPDE theory\label%
{section_unif_integr_rough_integrals}}

In situations where one performs a change of measure to an equivalent
measure on a path space, one often has to make sense of the exponential
moments of a stochastic integral, i.e. to show that%
\begin{equation}
E\left( \exp \left\{ \int G\left( X\right) \,dX+\int F\left( X\right)
\,dt\right\} \right)  \label{eqn_exp_integral}
\end{equation}%
is finite for a given process $X$ and some suitable maps $G$ and $F$. The
second integral is often trivially handled (say, when $F$ is bounded) and
thus take $F=0$ in what follows. Various situations in the literarture (e.g. 
\cite{H11}, \cite{CDFO}) require to bound (\ref{eqn_exp_integral}) uniformly
over a family of processes, say $\left( X^{\varepsilon }:\varepsilon
>0\right) $. We will see in this section that our results are perfectly
suited for doing this.

In the following, we study the situation of \cite{H11}, section 4. Here $%
\psi ^{\varepsilon }=\psi ^{\varepsilon }\left( t,x;\omega \right) $ is the
stationary (in time) solution to the damped stochastic heat equation with
hyper-viscosity of parameter $\varepsilon >0$,%
\begin{equation*}
d\psi ^{\epsilon }=-\epsilon ^{2}\partial _{xxxx}\psi ^{\epsilon
}\,dt+\left( \partial _{xx}-1\right) \psi ^{\epsilon }\,dt+\sqrt{2}\,dW_{t}
\end{equation*}%
where $W$ is space-time white noise, a cylindrical Wiener process over $%
L^{2}\left( \mathbb{T}\right) $ where $\mathbb{T}$ denotes the torus, say $%
\left[ -\pi ,\pi \right] $ with periodic boundary conditions. Following \cite%
{H11} we fix $t$, so that the "spatial" interval $\left[ -\pi ,\pi \right] $
plays the role of our previous "time-horizon" $\left[ 0,T\right] $. Note that%
$\mathbb{\,}x\mapsto \psi ^{\epsilon }\left( x,t\right) $ is a centred
Gaussian process on $\mathbb{T}$, with independent components and covariance
given by%
\begin{equation*}
E\left( \psi ^{\epsilon }\left( x,t\right) \otimes \psi ^{\epsilon }\left(
y,t\right) \right) =R^{\epsilon }\left( x,y\right) I=K^{\epsilon }\left(
x-y\right) I
\end{equation*}%
where $K_{\epsilon }\left( x\right) $ is proportional to 
\begin{equation*}
\sum_{k\in \mathbb{Z}}\frac{\cos \left( kx\right) }{1+k^{2}+\epsilon
^{2}k^{4}}.
\end{equation*}%
As was pointed out by Hairer, it can be very fruitful in a non-linear SPDE
context to consider $\lim_{\varepsilon \rightarrow 0}\psi ^{\varepsilon
}\left( t,\cdot ,\omega \right) $ as random spatial rough path. To this end,
it is stated (without proof) in \cite{H11} that the covariance of $\psi
^{\varepsilon }$ has fnite $\rho $-variation in 2D sense, $\rho >1$,
uniformly in $\varepsilon $. In fact, we can show something slightly
stronger. Following \cite{H11}, $\psi ^{\varepsilon }$ is $C^{1}$ in $x$ for
every $\varepsilon >0$, and can be seen as $p$-rough path, any $p>2$, when $%
\varepsilon =0$.

\begin{lemma}
\label{lemma_unif_1_var_hyperviscosity}The map$\mathbb{\ T}^{2}\ni \left(
x,y\right) \mapsto R^{\epsilon }\left( x,y\right) $ has finite $1$-variation
in 2D sense, uniformly in $\varepsilon $. That is, 
\begin{equation*}
M:=\sup_{\varepsilon \geq 0}\,V_{1-var}\left( R^{\varepsilon };\mathbb{T}%
^{2}\right) <\infty .
\end{equation*}
\end{lemma}

\begin{proof}
By lower semi-continuity of variation norms under pointwise convergence, it
suffices to consider $\varepsilon >0$. (Alternatively, the case $\varepsilon
=0$ is treated explicitly in \cite{H11}). We then note that%
\begin{equation*}
\sum_{k\in \mathbb{Z}}\frac{\cos \left( kx\right) }{1+\epsilon ^{2}k^{2}}=%
\frac{\pi \cosh \left( \frac{1}{\epsilon }\left( |x|-\pi \right) \right) }{%
\epsilon \sinh \left( \frac{\pi }{\epsilon }\right) }
\end{equation*}%
in $L^{2}\left( \mathbb{T}\right) $ as may be seen by Fourier expansion on $%
\left[ -\pi ,\pi \right] $ of the function $x\mapsto \cosh \left( \frac{1}{%
\epsilon }\left( |x|-\pi \right) \right) $. Since $\left\vert \partial
_{x,y}R_{\epsilon }\left( x,y\right) \right\vert =\left\vert K_{\epsilon
}^{\prime \prime }\left( x-y\right) \right\vert $ we have%
\begin{equation*}
V_{1-var}\left( R^{\varepsilon };\mathbb{T}^{2}\right) =\int_{\mathbb{T}%
^{2}}\left\vert \partial _{x,y}R_{\epsilon }\left( x,y\right) \right\vert
\,dx\,dy=\int_{\mathbb{T}^{2}}\left\vert K_{\epsilon }^{\prime \prime
}\left( x-y\right) \right\vert \,dx\,dy
\end{equation*}%
On the other hand, 
\begin{eqnarray*}
\left\vert K_{\epsilon }^{\prime \prime }\left( x\right) \right\vert  &\leq
&\left\vert \sum_{k\in \mathbb{Z}}\left( \frac{1}{1+\epsilon ^{2}k^{2}}-%
\frac{k^{2}}{1+k^{2}+\epsilon ^{2}k^{4}}\right) \cos \left( kx\right)
\right\vert +\frac{\pi \cosh \left( \frac{1}{\epsilon }\left( |x|-\pi
\right) \right) }{\epsilon \sinh \left( \frac{\pi }{\epsilon }\right) } \\
&=&\left\vert 1+\sum_{k\neq 0}\frac{1}{k^{2}}\frac{\cos \left( kx\right) }{%
\left( 1+\epsilon ^{2}k^{2}\right) \left( 1/k^{2}+1+\epsilon
^{2}k^{2}\right) }\right\vert +\frac{\pi \cosh \left( \frac{1}{\epsilon }%
\left( |x|-\pi \right) \right) }{\epsilon \sinh \left( \frac{\pi }{\epsilon }%
\right) } \\
&\leq &1+\sum_{k\neq 0}\frac{1}{k^{2}}+\frac{\pi \cosh \left( \frac{1}{%
\epsilon }\left( |x|-\pi \right) \right) }{\epsilon \sinh \left( \frac{\pi }{%
\epsilon }\right) } \\
&\leq &1+\frac{\pi ^{2}}{3}+\frac{\pi \cosh \left( \frac{1}{\epsilon }\left(
|x|-\pi \right) \right) }{\epsilon \sinh \left( \frac{\pi }{\epsilon }%
\right) }.
\end{eqnarray*}%
Hence%
\begin{equation*}
\int_{\mathbb{T}^{2}}\left\vert K_{\epsilon }^{\prime \prime }\left(
x-y\right) \right\vert \,dx\,dy\leq \left( 2\pi \right) ^{2}\left( 1+\frac{%
\pi ^{2}}{3}\right) +\frac{\pi }{\epsilon \sinh \left( \frac{\pi }{\epsilon }%
\right) }\int_{\mathbb{T}^{2}}\cosh \left( \frac{1}{\epsilon }\left(
|x-y|-\pi \right) \right) \,dx\,dy.
\end{equation*}%
We leave to the reader to see that the final integral is bounded,
independent of $\varepsilon $. For instance, introduce $z=x-y$ as new
variable so that only%
\begin{eqnarray*}
&&\frac{\pi }{\epsilon \sinh \left( \frac{\pi }{\epsilon }\right) }%
\int_{-2\pi }^{2\pi }\cosh \left( \frac{1}{\epsilon }\left( \left\vert
z\right\vert -\pi \right) \right) \,dz \\
&=&4\frac{\pi }{\epsilon \sinh \frac{\pi }{\epsilon }}\int_{0}^{\pi }\cosh
\left( \frac{z}{\epsilon }\right) \,dz
\end{eqnarray*}%
needs to be controlled. Using $\cosh \approx \sinh \approx \exp $ for large
arguments\ (or integrating explicitly ...) we get 
\begin{eqnarray*}
\frac{1}{\epsilon \sinh \left( \frac{\pi }{\epsilon }\right) }\int_{0}^{\pi
}\cosh \left( \frac{z}{\epsilon }\right) \,dz\, &\approx &\int_{0}^{\pi }%
\frac{\exp \left( \frac{z}{\epsilon }\right) }{\epsilon \exp \left( \frac{%
\pi }{\epsilon }\right) }\,dz \\
&=&\frac{\exp \left( \frac{\pi }{\varepsilon }\right) -1}{\exp \left( \frac{%
\pi }{\epsilon }\right) } \\
&\leq &1
\end{eqnarray*}
\end{proof}

\bigskip

We then have the following sharpening of \cite{H11},\ Theorem 5.1.

\begin{theorem}
\bigskip Fix $\gamma >2$ and $p\in \left( 2,\gamma \right) $. Assume $%
G=\left( G_{i}\right) _{i=1,\dots ,d}$ is a collection of $Lip^{\gamma
-1}\left( \mathbb{R}^{d},\mathbb{R}^{e}\right) $ maps. Then for some
constant $\eta =\eta \left( \gamma ,p,\left\Vert G\right\Vert _{Lip^{\gamma
-1}}\right) >0$ we have the uniform estimate%
\begin{equation*}
\sup_{t\in \lbrack 0,\infty )}\sup_{\varepsilon \geq 0}E\left\{ \exp \left(
\eta \left\vert \int_{\mathbb{T}}G\left( \psi ^{\epsilon }\left( x,t\right)
\right) \,d_{x}\psi ^{\epsilon }\left( x,t\right) \right\vert ^{2}\right)
\right\} <\infty .
\end{equation*}%
(When $\varepsilon >0$, $\psi ^{\epsilon }$ is known to be $C^{1}$ in $x$ so
that we deal with Riemann--Stieltjes integrals, when $\varepsilon =0$, the
integral is understood in rough path sense.)
\end{theorem}

\begin{proof}
By stationarity in time of $\psi ^{\epsilon }\left( \cdot ,t\right) $,
uniformity in $t$ is trivial. Note that the Riemann--Stieltjes integral 
\begin{equation*}
\int_{\mathbb{T}}G\left( \psi ^{\epsilon }\left( x,t\right) \right)
\,d_{x}\psi ^{\epsilon }\left( x,t\right) 
\end{equation*}%
can also be seen as rough integral where the integrator is given by the
"smooth" rough path $\left( \psi ^{\varepsilon },\int \psi ^{\varepsilon
}\otimes d_{x}\psi ^{\varepsilon }\right) $ when $\varepsilon >0$. For $%
\varepsilon =0$, the above integral is a genuine rough integral, the
existence of a canonical lift of $\psi ^{0}\left( \cdot ,t\right) $ to a
geometric rough path is a standard consequence (cf. \cite%
{friz-victoir-2007-gauss, FV10}) of finite $1$-variation of $R^{0}$, the
covariance function of $\psi ^{0}\left( \cdot ,t\right) $. After these
remarks,    
\begin{equation*}
\sup_{\varepsilon \geq 0}E\left\{ \exp \left( C\left\vert \int_{\mathbb{T}%
}G\left( \psi ^{\epsilon }\left( x,t\right) \right) \,d_{x}\psi ^{\epsilon
}\left( x,t\right) \right\vert ^{2}\right) \right\} <\infty 
\end{equation*}%
is an immediate application of\ Lemma \ref{lemma_unif_1_var_hyperviscosity}
and Proposition \ref{prop_tail_estimates_sdes_gaussian}.
\end{proof}

\begin{acknowledgement}
It is a pleasure to thank M. Hairer, H. Weber and T. Lyons for related
discussions. P.K. Friz has received funding from the European Research
Council under the European Union's Seventh Framework Programme
(FP7/2007-2013) / ERC grant agreement nr. 258237. S. Riedel is supported by
an IRTG (Berlin-Zurich) PhD-scholarship.
\end{acknowledgement}


\begin{thebibliography}{99}
\bibitem{BH} Baudoin, F.; Hairer, M.: A version of H\"{o}rmander's theorem
for the fractional Brownina motion, Probab. Theory Related Fields 139,
373--395, 2007

\bibitem{CF} Cass, T.; Friz, P.: Densities for rough differential equations
under H\"{o}rmander's condition, Ann. of Math. (2), 171(3):2115--2141, 2010

\bibitem{CFV} Cass, T.; Friz, P.; Victoir, N.: Non-degeneracy of Wiener
functionals arising from rough differential equations, Trans. Amer. Math.
Soc., 361(6):3359--3371, 2009

\bibitem{CLL} Cass, T.; Litterer, C.; Lyons, T.: Integrability estimates for
Gaussian rough differential equations, arXiv:1104.1813v4, 2011

\bibitem{CL} Cass, T.; Lyons, T.: Evolving communities and individual
preferences, Preprint, 2010

\bibitem{CQ02} Coutin, L.; Qian, Z.: Stochastic analysis, rough path
analysis and fractional Brownian motion, Probab. Theory Related Fields 122,
108--140, 2002

\bibitem{CDFO} Crisan, D.; Diehl, J.; Friz, P.; Oberhauser, H.: Robust
filterin: correlated noise and multidimensional observation,
arXiv:1201.1858v1, 2012

\bibitem{D10} Driscoll, P.: Smoothness of Density for the Area Process of
Fractional Brownian Motion, arXiv:1010.3047v1, 2010

\bibitem{FO09} Friz, P.; Oberhauser, H.: Rough paths limits of the
Wong-Zakai type with a modified drift term, J. Funct. Anal.,
256(10):3236-3256, 2009

\bibitem{FO10} Friz, P.; Oberhauser, H.: A generalized Fernique theorem and
applications, Proceedings of the American Mathematical Society, 138,
3679-3688, 2010

\bibitem{friz-victoir-2007-gauss} Friz, P.; Victoir, N.: Differential
equations driven by {G}aussian signals, Annales de l'Institut Henri Poincar%
\'{e} (B) Probability and Statistics, Vol. 46, No. 2, 369--413, 2010

\bibitem{FV10} Friz, P.; Victoir, N.: Multidimensional Stochastic Processes
as Rough Paths, Cambridge Studies in Advanced Mathematics Vol 120, Cambridge
University Press, 2010

\bibitem{H11} Hairer, M.: Rough stochastic PDEs, Comm. PureAppl. Math. 64,
no.11, 1547--1585, 2011

\bibitem{HP11} Hairer, M.; Pillai, N. S.: Regularity of Laws and Ergodicity
of Hypoelliptic SDEs Driven by Rough Paths, arXiv:1104.5218v1, 2011

\bibitem{I10} Inahama, Y.: A moment estimate of the derivative process in
rough path theory, arXiv:1007.4651v1, 2010

\bibitem{Le} Ledoux, M.: Isoperimetry and Gaussian Analysis, Lectures on
Probability Theory and Statistics (Saint-Flour, 1994), Vol. 1648 of Lecture
Notes in Mathematics, p. 165-294, Springer, 1996

\bibitem{LQZ} Ledoux, M.; Qian, Z.; Zhang, T.: Large deviations and support
theorem for diffusion processes via rough paths, Stoch. Proc. Appl. 102, no.
2, 265-283, 2002

\bibitem{Lej09} Lejay, A.: On rough differential equations, Electronic
Journal of Probability [Online], Vol. 14, 2009

\bibitem{L98} Lyons, T.: Differential equations driven by rough signals,
Rev. Mat. Iberoamericana 14, no. 2, 215--310, 1998

\bibitem{LQ02} Lyons, T.; Qian, Z.: System Control and Rough Paths,\ Oxford
University Press, 2002

\bibitem{MR2314753} Lyons, T.; Caruana, M.; L{\'{e}}vy, T.: Lectures on
Probability Theory and Statistics (Saint-Flour, 1996), Vol. 1908 of Lecture
Notes in Mathematics, Springer, 2007

\bibitem{LZ99} Lyons, T.; Zeitouni, O.: Conditional exponential moments for
iterated Wiener integrals, Ann. Probab. 27, no. 4, 1738-1749, 1999

\bibitem{Ma} Malliavin, P.: Stochastic Analysis, Volume 313 of Grundlehren
der Mathematischen Wissenschaften [Fundamental Principles of Mathematical
Sciences], Springer, 1997

\bibitem{MS} Millet, A.; Sanz-Sol\'{e}, M.: Large deviations for rough paths
of the fractional Brownian motion, Annales de l'Institut Henri Poincar\'{e}
(B) Probability and Statistics, Vol. 42, 245-271, 2006
\end{thebibliography}
\end{document}